\documentclass[10pt,a4paper,english]{article}

\usepackage[english]{babel}
\usepackage[latin1]{inputenc} 
\usepackage{amscd, amssymb, mathrsfs,amsfonts, amsthm} % Setzen mathematischer Formeln
\usepackage{mathtools}
\usepackage{ellipsis} % korrigiert Aussehen von „…“
\usepackage{enumitem} % mehr Optionen für Aufzählungen
\usepackage{multirow} % Zusammenfügen von Zellen in Tabellen
\usepackage[shortcuts]{extdash} % mehr Optionen für Worttrennung
\usepackage[version=latest]{pgf}
\usepackage{bbm} % 1 mit Doppelstrich
\usepackage{todonotes}

\usepackage{parskip} %paragraph indent looks bad
\usepackage[square,numbers]{natbib} %citation using only numbers
\usepackage{tikz}
\usetikzlibrary{arrows} %figures
\usepackage{rotating}
\usepackage{hyperref} % Verlinkungen im Dokument

\newcommand{\cross}{\mathop{\backslash\mathllap{/}}}
\newcommand{\dbars}{\mathop{||}}

\newcommand{\tikzcross}[4]{
	\draw[color=#3] (#1,#2)+(0,0) -- +(2,0.2);
	\draw[color=#4] (#1,#2)+(0,.2) -- +(2,0);
	\draw[color=white] (#1,#2)+(0,0) -- +(0,0.2);
	\draw[color=white] (#1,#2)+(2,0) -- +(2,0.2);
}

\newcommand{\tikzdbars}[4]{
	\draw[color=#3] (#1,#2)+(0,0) -- +(2,0);
	\draw[color=#4] (#1,#2)+(0,.2) -- +(2,0.2);
	\draw[color=white] (#1,#2)+(0,0) -- +(0,0.2);
	\draw[color=white] (#1,#2)+(2,0) -- +(2,0.2);
}

\definecolor{myblue}{RGB}{93.9,129.2,181}
\definecolor{mypurple}{RGB}{120,94,240}
\definecolor{myred}{RGB}{220,38,127}
\definecolor{myorange}{RGB}{254,97,0}
\definecolor{myyellow}{RGB}{224.6,155.8,36.2}
\definecolor{graphblue}{RGB}{93.9,129.2,181}%{0.368417,0.506779,0.709798}
\definecolor{graphbrown}{RGB}{224.6,155.8,36.2}%{0.880722,0.611041,0.142051}

\let\oldsqrt\sqrt
\def\sqrt{\mathpalette\DHLhksqrt}
\def\DHLhksqrt#1#2{%
\setbox0=\hbox{$#1\oldsqrt{#2\,}$}\dimen0=\ht0
\advance\dimen0-0.2\ht0
\setbox2=\hbox{\vrule height\ht0 depth -\dimen0}%
{\box0\lower0.4pt\box2}}
 
\newcommand{\ord}{\mathop{\operatorname{ord}}}

\newcommand{\pred}{\mathop{\operatorname{pred}}}

\newcommand{\supp}{\mathop{\operatorname{supp}}}

\newcommand{\betaplus}{\mathop{\beta^*}}
\newcommand{\Tbeta}{\mathop{%[0,\beta)_\per
		\mathbb T_{\beta}}}

\newcommand{\dreg}{$d$-ary }
\newcommand{\FK}{\mathop{F_K}}
\newcommand{\FKb}{\mathop{\bar F_K}}

\numberwithin{equation}{section}
\theoremstyle{plain}
\newtheorem{thrm}{Theorem}[section]
\newtheorem{prop}[thrm]{Proposition}

\newtheorem{lem}[thrm]{Lemma}

\theoremstyle{definition}

\newtheorem{ex}[thrm]{Example}
\newtheorem{rem}[thrm]{Remark}
\newtheorem{cprob}[thrm]{Combinatorial Problem}

\begin{document}

% Rescales the size of section and subsection titles and distances around them (makes everything a little bit smaller)
%\makeatletter
%\renewcommand\section{\@startsection {section}{1}{\z@}%
%	{-0.5ex \@plus -2ex \@minus -.2ex}%
%	{0.3ex \@plus.2ex}%
%	{\normalfont\Large\bfseries}}% from \Large
%\renewcommand\subsection{\@startsection {subsection}{1}{\z@}%
%	{-0.5ex \@plus -2ex \@minus -.2ex}%
%	{0.3ex \@plus.2ex}%
%	{\normalfont\large\bfseries}}% from \Large
%\makeatother	

	\title{Sharp phase transition for random loop models on trees}
	\author{Volker Betz\footnote{betz@mathematik.tu-darmstadt.de}, Johannes Ehlert\footnote{ehlert@mathematik.tu-darmstadt.de}, Benjamin Lees\footnote{benjamin.lees@bristol.ac.uk}, Lukas Roth\footnote{lroth@mathematik.tu-darmstadt.de}
		}
	\date{ %}\address{
		\small Technische Universit\"{a}t Darmstadt, Germany}
	\maketitle
	\begin{abstract} We investigate the random loop model on the $d$-ary tree. 
		For $d \geq 3$, we establish %the existence of 
		a (locally) sharp phase transition for the existence of infinite loops. Moreover, we derive rigorous bounds that in principle allow to determine the value of the critical parameter with arbitrary precision. 
		Additionally, we prove the existence of an asymptotic expansion for the critical parameter in terms of $d^{-1}$. The corresponding coefficients can be determined in a schematic way and we calculate them up to order $6$.
	\end{abstract}
	
	\section{Introduction} \label{sec:intro}
Let $G=(V,E)$ be an undirected (simple) graph and let 
$\Tbeta:=\mathbb R / \beta \mathbb Z$ be the one-dimensional torus with 
length $\beta>0$. A {\em link configuration} on  $E \times \Tbeta$ is a 
family $X=(X^{e, \star})_{e \in E, \star \in \{\cross, \dbars\}}$ of measures on $\Tbeta$, such that
$X^{e,\cross}+X^{e,\dbars}$ is a simple and finite atomic measure on $\Tbeta$ for each $e$, i.e. it is a finite sum of Dirac measures $\delta_{t_i}$ with $t_i \neq t_j$ when $i \neq j$.
%Moreover, $X^{e,\cross}+
The atoms of $X^{e,\star}$ are called {\em links} and each link of $X$ is specified by a triple $(e,t,\star)$, where $\star\in \{ \cross, \dbars\}$ is the type and $t$ the position/time of the link on the edge $e$.

Each link configuration induces a {\em loop configuration}, which is a 
collection of open subsets of the set 
$V \times \Tbeta$. The rigorous definition of the map from a link 
configuration to a loop configuration, which will be given shortly, 
is a bit technical; its essence however can be conveniently 
grasped from Figure \ref{fig:loop-definition}: A link of type $\cross$ on 
an edge connects those regions on $\Tbeta$ on the two 
vertices adjacent to its edge that are on opposing sides of its position, 
while a link of type $\dbars$ connects regions on the same side of its 
position. Regions on the same vertex are always separated by links 
on adjacent edges.
After extension by transitivity, this yields a partition of $V \times 
\Tbeta$ into the closed set $\{(x,t): x \in V, t \in \supp X^{e,\star} 
\text{ for some } e \ni x, \star \in \{\cross,\dbars\} \}$, and the open 
sets of mutually connected points (the {\em loops}). 

The relevant quantity for loop models is the size of typical 
loops (in our case measured in the number of visited vertices, 
although the arc length 
is also a conceivable quantity of interest) when the link configuration is random. 
More precisely, the question is whether a given 
family of loop models has a percolation phase transition in the parameter
$\beta$, i.e.\ whether (for an infinite graph) the probability that a 
given fixed vertex is contained in an infinite loop is positive for some $
\beta$ and zero for others. 
The apparently simplest case is $G = \mathbb Z^d$, 
$X^{e,\dbars} = 0$ for all edges $e$ and the $X^{e,\cross}$ are iid Poisson point
processes of rate $1$. While numerical results 
\cite{Barp}
strongly suggest the existence of a phase 
transition, on a rigorous level the question is completely 
open in this case.

The main difficulty in the loop model is the 
lack of 
monotonicity, i.e. more links do not 
necessarily mean longer loops. This can already be 
seen in Figure \ref{fig:loop-definition}: removing 
one of the links between the two middle vertices 
merges the red and blue loop into one%, enlarging the number of vertices that the red loop visits
. Moreover, 
local changes of the loop configuration can connect 
or disconnect intervals in very different regions 
of $G$, so the model is highly non-local in this 
sense. These two obstacles have so far prevented 
the development of efficient tools to investigate 
percolation on loop models on most graphs, leading 
to a relative scarcity of results; however, a few results exist, and we will 
review them now.

	\begin{figure}
		\begin{tikzpicture}[scale=0.75]
		%Vertices at t=0
		\node[circle,scale=0.3,fill=black,draw] at (0,0) (x1) {};
		\node[circle,scale=0.3,fill=black,draw] at (2,0) (x2) {};
		\node[circle,scale=0.3,fill=black,draw] at (4,0) (x3) {};
		\node[circle,scale=0.3,fill=black,draw] at (6,0) (x4) {};
		%Link positions
		%	\node[inner sep=0pt] at (0,1.9) (l1) {};
		%	\node[inner sep=0pt] at (2,1.2) (l2) {};
		%	\node[inner sep=0pt] at (4,3.1) (l3) {};
		%	\node[inner sep=0pt] at (4,0.6) (l4) {};
		%	\node[inner sep=0pt] at (2,2.3) (l5) {};
		%Edges connecting these vertices
		\draw[dotted] (x1) -- (x2) -- (x3) -- (x4);
		%Coordinate axes
		\draw[->]  (-.5,-.5) -- node[below] {$G=(V,E)$} +(1.5,0) ;
		\draw[->] (-.5,-.5) -- +(0,5) node[above] {position};
		\draw (-.6,0) node[left] {$0$} -- (-.4,0);
		\draw (-.6,4) node[left] {$\beta$} -- (-.4,4);
		%Draw events
		\tikzcross{0}{1.9}{black}{black};
		\tikzdbars{2}{1.2}{black}{black};
		\tikzdbars{2}{2.3}{black}{black};
		\tikzcross{4}{3.1}{black}{black};
		\tikzdbars{4}{0.6}{black}{black};
		\end{tikzpicture}
		\hfill 
		\begin{tikzpicture}[scale=0.75]
		%Vertices at t=0
		\node[circle,scale=0.3,fill=black,draw] at (0,0) (x1) {};
		\node[circle,scale=0.3,fill=black,draw] at (2,0) (x2) {};
		\node[circle,scale=0.3,fill=black,draw] at (4,0) (x3) {};
		\node[circle,scale=0.3,fill=black,draw] at (6,0) (x4) {};
		%Link positions
		%	\node[inner sep=0pt] at (0,1.9) (l1) {};
		%	\node[inner sep=0pt] at (2,1.2) (l2) {};
		%	\node[inner sep=0pt] at (4,3.1) (l3) {};
		%	\node[inner sep=0pt] at (4,0.6) (l4) {};
		%	\node[inner sep=0pt] at (2,2.3) (l5) {};
		%Edges connecting these vertices
		\draw[dotted] (x1) -- (x2) -- (x3) -- (x4);
		%Coordinate axes
		\draw[->]  (-.5,-.5) -- node[below] {$G=(V,E)$} +(1.5,0) ;
		\draw[->] (-.5,-.5) -- +(0,5) node[above] {position};
		\draw (-.6,0) node[left] {$0$} -- (-.4,0);
		\draw (-.6,4) node[left] {$\beta$} -- (-.4,4);
		%Draw coloured loops
		\draw[color=myred] (x1) -- +(0,1.9) -- +(2,2.1) -- +(2,2.3) -- +(4,2.3) -- +(4,1.4) -- +(2,1.4) -- +(2,1.9) -- +(0,2.1) -- +(0,4);
%		\path[draw, color=green] (x1) -- ++(0,1.9) edge[dashed] ++(2,0.2);
		\draw[color=myblue] (x2) -- +(0,1.2) -- +(2,1.2) -- +(2,0.8) -- +(4,0.8) -- +(4,3.1) -- +(2,3.3) -- +(2,4);
		\draw[color=myblue] (x3) -- +(0,0.6) -- +(2,0.6) -- +(2,0);
		\draw[color=myblue] (x2)+(0,4) -- +(0,2.5) -- +(2,2.5) -- +(2,3.1) -- +(4,3.3) -- +(4,4);
		\end{tikzpicture}
		\caption{Example of a small finite graph $G$ and a link configuration $X$ (left) leading to the two depicted loops (right, {\color{myred}{red}} and {\color{myblue}{blue}}).}
		\label{fig:loop-definition}
	\end{figure}
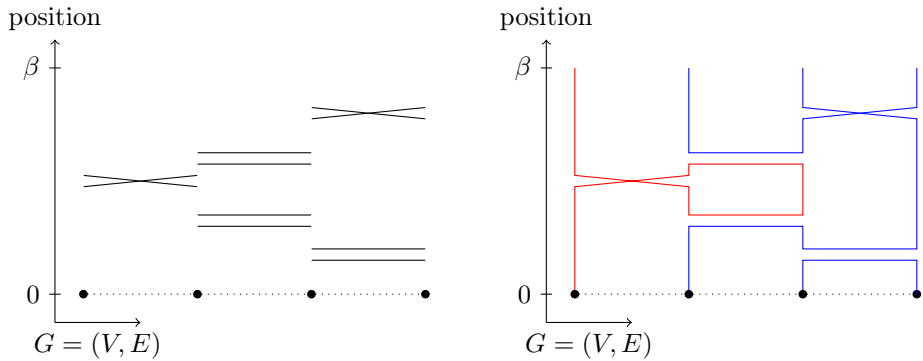
	
	In the context of probability theory, the loop 
	model goes back to the random stirring process, introduced by Harris \cite{Har-72}. This process $(\sigma_t)_{t\in [0,\beta]}$ of permutations on $V$ corresponds to the random loop model mentioned above, i.e.\ with $X^{e,\dbars} = 0$ and $X^{e,\cross}$ iid Poisson point processes. 
Namely, given a link configuration $X$ and setting 
$\sigma_0$ to be the identity permutation, we % $\sigma_0$ and 
	increase time $t$ and if there is a link on an edge $\{x,y\}$ at the position $t$ that we currently consider, we compose $\sigma_{t-}$ with the transposition of $x$ and $y$. It is easy to see that two vertices $x$ and $y$ are contained within the same cycle of $\sigma_\beta$ iff $(x,0)$ and $(y,0)$ share a loop.\\
	Note that, on arbitrary connected graphs with bounded degree, the critical parameter $\beta_c$ is strictly larger than the percolation threshhold for edges carrying at least one link \cite{Muehlbacher}.
	For the random stirring model on the (finite) complete graph, the phase transition occurs at $\beta_c=|V|^{-1}$ in the limit of $|V|\to \infty$, see e.g. \cite{Ber-11, BG-15}. Moreover, for a time-discrete model where one link occurs at each step and for time-scales above a critical value corresponding to $\beta_c=|V|^{-1}$ in our setting, Schramm showed in % Schramm
	\cite{Schramm05} that the distribution of cycle sizes within the 
	%macroscopic percolation 
	giant
	component converges (after renormalisation) to a Poisson-Dirichlet distribution of parameter $1$. In \cite{BKLM-18}, this result has been extended to include links of type $\dbars$, too.\\
	Apart from the complete graph and the $2$-dimensional Hamming graph \cite{SenMil-16}, another graph for which progress has been made in the context of the random stirring model is  the \dreg tree. Angel \cite{Ang-03} showed the existence of two different phases for $d\geq 4$, and the existence of infinite cycles for $\beta \in (d^{-1}+\tfrac{13}{6} d^{-2},\ln(3))$ 
	in the asymptotic regime $d \to \infty$. 
Hammond then showed in \cite{Ham-13} that (for $d\geq 2$) there is a value $\beta_0$ above which $\sigma_\beta$ contains infinite cycles and that for $d \geq 55$, one may chose $\beta_0=101 d^{-1}$. Furthermore and for even larger $d$, strict bounds for this critical parameter have been found in \cite{Ham-15} and it was shown that the transition from finite to infinite cycles is sharp. In the recent work of Hammond and Hegde \cite{Ham-18}, these bounds have been proven to hold for $d \geq 56$ while even including links of type $\dbars$. 
	Moreover, Bj\"{o}rnberg and Ueltschi \cite{BU-18-1} 
determined the critical parameter $\beta_c$ of the loop model %on the \dreg tree
	up to second order in $d^{-1}$ as $d\to \infty$. The reader should note that 
	the majority of the above results rely on graph degrees being comparatively large, or are even just asymptotic in them. 
	
	In the present paper we significantly improve the existing 
	results for $d$-ary trees and achieve a rather 
	complete picture of 
	the random loop model in these cases. We focus on the 
	case where  
	the $(X^{e,\star})$ are iid Poisson point processes, but it 
	should be clear how our method extends to other families of
	independent point processes. 
	While a simple percolation argument shows that almost surely there are no infinite loops for $\beta \leq d^{-1}$, our methods aim at the critical region $d^{-1} < \beta \leq d^{-1/2}$. 
	In 
	Theorem \ref{thrm:sharp-phase-transition} below, we establish the 
	existence of a sharp phase transition for all 
	$d \geq 3$ within this region, comparable results previously existed only up to $d^{-1}+2d^{-2}$ and for $d \geq 26$ \cite{Ham-18}. Additionally, in Theorem \ref{prop:betac-asymptotics},
	we provide an asymptotic expansion of the critical value 
	in powers of $1/d$, with coefficients depending on the
	parameter $u$ controlling the relative intensities of 
	the point processes $X^{e,\cross}$ and $X^{e,\dbars}$. 
	
	Our proofs rely on a natural idea: the central object are those edges that carry precisely one link.
	It is not difficult to see that at such edges a renewal event occurs: 
	Removing any edge $e$ splits the \emph{tree} $G$ into 
%	if $G$ ist a tree, removing any the edge $e$ splits $G$ into 
	two disconnected subtrees $G_1$ and $G_2$. 
%	If
	Thus, in the case that
	$e$ only carries a single link, and if the loop through that link is finite on $G_1$, say, then that loop has to pass through $e$ in both directions. Consequently, in this case 
	the loop structure on $G_2$ depends only on the link structure of $G_2$ and not on the link structure on $G_1$. This allows to construct renewal schemes that use single-link edges as `new roots'. 
	
	The first such renewal scheme was presented in the work of Angel \cite{Ang-03}. The paper uses a single-link edge $e = \{x,y\}$ as a renewal edge if the arrival time $t_e$ of its link is uncovered, meaning that 
	none of its siblings has a pair of links whose arrival times separate the time $t_e$ of the link on $e$ from the first time the loop meets the parent $x$ of $e$, 
	in the topolgy of the torus $\mathbb T_\beta$.  This guarantees that any loop arriving at the parent $x$ of $e$ either is already infinite or will eventually pass through $e$. The proof then consists in identifying 
	conditions under which infinitely many renewal edges exist with positive probability. The main limitation of this scheme is that being uncovered is a rather strong restriction on a single-link edge, 
	and that an un-interrupted chain of uncovered edges is needed from the root to infinity with positive probability. 
	Thus the criterion leads to conditions that are far from optimal. In particular they are only accurate to first order in $1/d$ for large $d$, and only work for $d \geq 4$. 
	
	Our approach is a more systematic one: we consider multilink-clusters, i.e., the finite subtrees of the infinite tree whose edges 
	all have more than one link, and use as renewal edges all single-link edges protruding from these subtrees that carry the loop entering the subtree 
	at its root. In comparison to the method of \cite{Ang-03}, this allows not only for covered single-link edges to be used, but it (in principle) allows us to cross any number of edges that have multiple links.
	
%	Of course, such an approach is useful only if we can determine what edges are renewal edges in a given subtree.\todo{Unclear} One of our results is that by a Galton-Watson argument, 
%	all we need to know is that the expected number of such edges emerging from a multilink-cluster is greater than one. This allows for a systematic treatment of the problem: we find an explicit formula for 
%	the expected number of renewal edges, for an (in principle) arbitrary finite multilink-cluster with an arbitrary number of links on their edges. Since this formula contains non-trivial combinatorics, 
%	we only get explicit expressions for multilink-clusters up to a certain size and link number. However, since the number of renewal edges clearly is a non-negative random variable, we obtain a lower bound to its 
%	expected value by just considering finitely many such subtrees, and we can improve this bound systematically by considering more of them. When combined a (relatively simple) upper bound for the number of 
%	expected renewal edges, we obtain very precise estimates for the value of $\beta_c$, and modulo the difficulty of evaluating the aforementioned combinatorial quantities, 
%	our method actually gives a full asymptotic expansion of  $\beta_c$ in $1/d$. 
	
	One limitation that our method does have is that it relies on a sufficiently high probability for the multilink-cluster to be \emph{finite}. This poses no limitation for $\beta \leq d^{-1/2}$ %1/\sqrt{d}$ 
	as this is below the percolation threshold for these clusters, but becomes an obstacle for higher $\beta$. Since $\beta_c \sim 1/d$ for large $d$, no problem appears
	in view of the asymptotic estimates for $\beta_c$, and in the regime of small $d$ our results are sufficient
%	ly sharp to allow for results in the previously unknown cases $d=3,4$. 
	to identify $\beta_c$ with high precision.
	However, for the proof that there is no phase of almost surely finite loops beyond $\beta_c$, we need to rely on results obtained by Hammond and Hegde \cite{Ham-18}. It is not inconcievable that this could be improved by suitable lower bounds for the expected number of renewal edges for an infinite (or very large) multilink-cluster, but such an investigation is beyond the scope of the current work and left for future investigations.

	Apart from random stirring, a strong motivation for 
	studying random loop models comes from their relation to quantum mechanical models. More precisely, in \cite{AizNa} and \cite{Toth-iLB} stochastic representations of the spin-$\tfrac{1}{2}$ quantum Heisenberg antiferromagnet and ferromagnet, respectively, were studied. 
	Recently, Ueltschi \cite{Ueltschi-13} introduced the random loop model as a common generalisation that interpolates between those representations and also includes a representation of the spin-$\tfrac{1}{2}$ XY model. For these representations, each link configuration receives a weight proportional to $\theta^{\# \text{loops}}$, so for $\theta \neq 1$, links on different edges are no longer independent. Also, the model cannot be directly defined on an infinite graph. Thus, it has to be constructed via an infinite volume limit. Physcially, 
	$\theta = 2$ is the most relevant case. The occurrence of infinite loops is then related to % the respective quantum models showing 
	non-decay of correlations for the quantum spin systems. Therefore, in order to see that these % quantum spin 
	systems %mechanical models 
	undergo a phase transition and to determine the critical inverse temperature $\beta_c$ at which it occurs%occuring within these models
	, one possibility is to investigate the different phases of the random loop model. 
	%Adapting the method reflection positivity from \cite{DLS-78}, Ueltschi showed the existence of a phase transition on $\mathbb Z^d$ for
	\\
	As it is the case in the random stirring model, the most interesting (but also apparently the most challenging) graph to study these models on is $\mathbb Z^d$. % the above papers considered $\mathbb Z^d$,
	Mathematical results exist for %other graphs such as %the complete graph \cite{BKLM-18,Schramm05}, 
	the complete graph \cite{Bj15,BFU18}, the $2$-dimensional Hamming graph \cite{AKM-18}, Galton-Watson trees \cite{BEL-18} and the \dreg tree \cite{%Ang-03,%BU-18-1,
		BU-18-2%,Ham-13,Ham-15,Ham-18
	}, again in the regime of high degrees. 
	Unfortunately, for $\theta \neq 1$, the weighted measures involve intricate correlations and the techniques of our paper do not directly apply.  
	
This paper is organized as follows: In Section \ref{sec:results},
we state our precise assumptions and results. In Section 
\ref{sec:exploration-scheme}, we introduce {\em exploration
schemes}, a recursive construction with a renewal structure that 
constitutes the core of our proof as it gives a Galton-Watson process whose survival is related to the event that the loop containing the root at time $0$ is infinite. 
This enables us to distinguish the phases by
considering the expected value for the first generation of this process 
%enables us to distinguish the phases 
and 
without much further work, we are 
then already able to establish a locally sharp phase transition
%prove Theorem \ref{thrm:sharp-phase-transition} %\ref{thrm:exploration-scheme-conditions} 
for all $d \geq 5$.
Afterwards, within Section \ref{sec:asymptotic-expansion}, we will turn our attention to the asymptotic expansion and on the way to its proof, we will discover sufficient (and computable) conditions for the two phases. 
Finally, in Section 
\ref{sec:combinatorics}, we will then establish the necessary computations that
enable us to push our results to $d=3$ and to calculate coefficients within the asymptotic expansion.
	
\section{Main results} \label{sec:results}	
		
We start by giving a proper definition of the map from link 
configurations to loop configurations. 
Suppose that 
	$X = (X^{e,\star})_{e \in E, \star \in \{ 
		{\cross}, {\dbars} \}}$ is a link configuration. 
We call $X$ {\em admissible} if $X^{e,\star}$ and $X^{e',\star'}$ 
	are mutually singular whenever $e \neq e'$ but 
	$e \cap e' \neq \emptyset$, and also when $e = e'$ and 
	$\star \neq \star'$. This guarantees that the construction of loops given below
	is well defined. When fixed link configurations are given, we will always assume that they are admissible, and that all our link-configuration-valued random variables will produce admissible link configurations almost surely. 
	
	% for the given link configuration. 
	Given an admissible link configuration $X$, a loop is an equivalence class of elements of $V \times \Tbeta$ 
	induced by the following \textit{connectedness relation}:
We equip $V$ with the discrete and $\Tbeta$ with the quotient topology and say that two points $(x_0,t_0)$ and %$(x_1,t_1)$ in \linebreak{} $V\times\Tbeta$ 
	$(x_1,t_1) \in V\times\Tbeta$
	are {\em connected} iff there is no link on an edge incident to $x_i$ at position $t_i$, $i=0,1$, and there is a piecewise continuous path %\linebreak{}
	$\Gamma =(\Gamma_1,\Gamma_2) \colon [0,1] \to V \times \Tbeta$ from $(x_0,t_0)$ to $(x_1,t_1)$ such that
	\begin{itemize}
		\item $\Gamma_2$ is continuous everywhere and differentiable at every point of continuity of $\Gamma$. Where the derivative $\Gamma_2'$ exists, its absolute value is % with $|\gamma_2'|$ being 
		a fixed constant.
		%	{\color{red}{\\Do we need to identify $\Tbeta$ with $[0,\beta)$?}}
		\item If $\Gamma$ is discontinuous at $s \in (0,1)$, then there is a link on $\{\Gamma_1(s-),\Gamma_1(s+)\}$ at position $\Gamma_2(s)$.
		\item For all links $(\{x,y\},t,\star)$ of $X$ such that 
		$\Gamma(s-)=(x,t)$ (or $\Gamma(s+)=(x,t)$) for some $s\in (0,1)$ we have $\Gamma(s+)=(y,t)$ (or $\Gamma(s-)=(y,t)$, respectively) as well as
		\begin{align*}
		\Gamma_2'(s+)=\left \{ \begin{array}{cl} + \Gamma_2'(s-) & \text{ if }\star=\cross, \\ - \Gamma_2'(s-) & \text{ if }\star=\dbars. \end{array}  \right.
		\end{align*}
	\end{itemize}
	Note that a loop $\gamma$ is by definition a subset of $V \times \Tbeta$. Nevertheless, in a slight abuse of notation, we write $x \in \gamma$ iff there is a $t \in \Tbeta$ with $(x,t) \in \gamma$. Similarly, we set
	\begin{align*}
	|\gamma|:=|\{x \in V : x \in \gamma \}|.
	\end{align*}
	Now that we have defined loops, let us fix our assumptions. 
	We write $T = (V,E)$ for the \dreg tree with root $r \in V$, i.e. the tree where 
	each vertex has $d$ `children' and (except for $r$) one `parent'.
	We assume that the link configuration is given
	by an independent family 
	$(X^{e,\star})_{e \in E, \star \in \{\cross,\dbars\} }$ 
	of homogeneous Poisson point processes, where for each $e \in E$, 
	$X^{e, \cross}$ has rate $u \in [0,1]$ and
	$X^{e, \dbars}$ has rate $1-u$. Under these 
	assumptions, we have: 
		
	\medskip
	\begin{thrm}[Existence and local sharpness of the phase transition] \label{thrm:sharp-phase-transition} $~$\\
		Let $\gamma_T$ be the loop on $T$ containing $(r,0)$. Then for all $d \geq 3$ and for all $u \in [0,1]$ there exist
		$\betaplus \geq d^{-1/2}$ and $\beta_c \in (0,\betaplus)$
%		$\beta_c>0$ and $\betaplus>\beta_c$ 
		such that
			\begin{enumerate}[label=\textnormal{(\textit{\roman*})}]
				\item $|\gamma_T|<\infty$ almost surely for all $\beta \leq \beta_c$,
				\item $|\gamma_T|=\infty$ with positive probability for all $\beta \in (\beta_c,\betaplus)$.
			\end{enumerate}
%			Moreover, $\betaplus\geq d^{-1/2} %\tfrac{1}{\sqrt{d}}
%			$ for all $d \geq 3$ and $\betaplus=\infty$ for $d \geq 16$.
	\end{thrm}	

	Note that, for $d=1$, there is no phase transition since $|\gamma_T|<\infty$ almost surely (non-zero probability of empty edges). Moreover, the case $d=2$ technically is accessible with our method. However, it would take much more computational effort to prove a similar statement in this case, see Remark \ref{rem:d2transition}. Furthermore note that a re-entry into the phase of finite loops for $\beta > d^{-1/2}$ is quite implausible. 
	Nevertheless, we cannot exclude this behaviour as our method %rule out such a re-entry since our approach 
	is tailored %to work in the regime 
	for $\beta$ up to $d^{-1/2}$. Still, in combination with \cite[Proposition 1.2 (2),(4)]{Ham-18}, Theorem \ref{thrm:sharp-phase-transition} suffices to
	show 
%	we may combine our  to find 
	that there is no re-entry and that the phase transition is thus globally sharp for all $d \geq 16$, therefore improving the previously known lower bound of $d \geq 56$ from \cite{Ham-18}.

	\begin{figure}
	\resizebox{\textwidth}{!}{
		\begin{tabular}{p{0.5\textwidth} p{0.5\textwidth}}
			\includegraphics[height=165px]{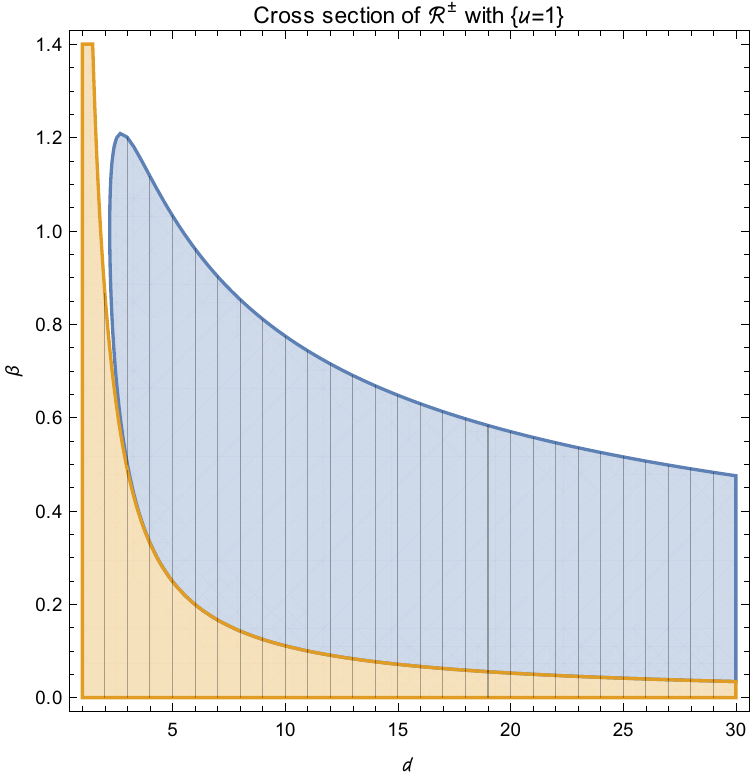} &
			\includegraphics[height=166px]{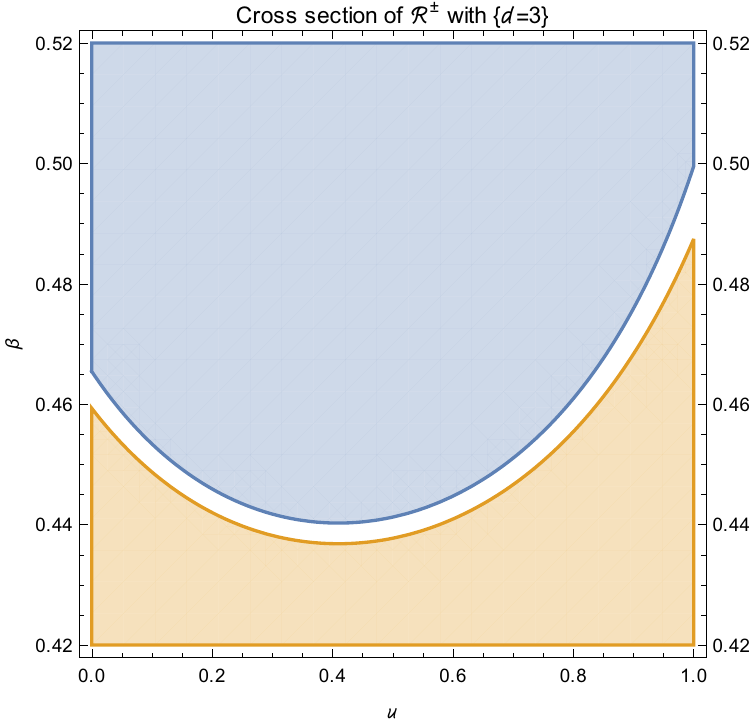}
		\end{tabular}
	}
	\caption{Regions $\mathcal R_5^\pm$ of parameters $(\beta,d,u)$ where we can guarantee that $\gamma_T$ is infinite with positive probability (upper/blue region $\mathcal R^+_5$) and that $\gamma_T$ is finite almost surely (lower/sandybrown region $\mathcal R^-_5$), respectively. See (\ref{eq:Rplusdef}) and (\ref{eq:Rminusdef}) for a precise definition of these regions.
	}
	\label{fig:phase-subregions}
\end{figure}
	
	In addition to establishing a phase transition, the tools we develop %to establish the phase transition 
	also yield an equation in $\beta$ that is solved by $\beta_c$, compare Proposition \ref{prop:betac-characterisation} and (\ref{eq:EM1rep}). We may then approximate its terms systematically to find sharp bounds on the critical parameter $\beta_c$ for every $d \geq 3$.
	These estimations rely on solving a certain combinatorial problem associated to finite edge-weighted trees, giving implicit conditions about the phase regions. 
	Instead of providing explicit but imprecise bounds for $\beta$ (which would also be possible, cf. (\ref{eq:betacestimationexplicit})), we rather check whether one of these implicit conditions is satisfied. Thereby,
%	In particular, 
	we obtain a region of parameters $(\beta,d,u)$ where $\gamma_T$ is infinite with positive probability (blue region in Figure \ref{fig:phase-subregions}) and a region where it is finite almost surely (sandybrown region in Figure \ref{fig:phase-subregions}). The critical parameter $\beta_c$ thus lies within the small (white) gap between these regions. 
	For more details on these implicit conditions and the corresponding combinatorial problem, we refer to Lemma \ref{lem:EM1-estimates} 
%	as well as
	and
%	The calculations 
%	leading to Figure \ref{fig:phase-subregions} can be found in 
	Section \ref{sec:combinatorics}.

	A further analysis of the terms within the determining equation for $\beta_c$  yields the following result.
	
	\medskip
	\begin{thrm}[Asymptotic expansion of $\beta_c$]\label{prop:betac-asymptotics} ~\\
	There exist polynomials $\alpha_0, \alpha_1, \alpha_2,
	\ldots$ such that for any $K \in \mathbb N_0$ the critical parameter is asymptotically given by
	\begin{align} \label{eq:betacexpansion}
	\beta_c = \sum_{k=1}^{K+1} \frac{\alpha_{k-1}(u)}{d^k} + \mathcal O(d^{-(K+2)})
	\end{align}
	as $d \to \infty$.	
	\end{thrm} 
	
In fact, we know somewhat more than just the existence of the polynomials $\alpha_0,\ldots, \alpha_{K}$: the degree of $\alpha_k$ is at most $2k$ and each $\alpha_k$ is explicitly given in terms of $\alpha_0,\ldots,\alpha_{k-1}$ as well as derivatives of a function $F_k$, see (\ref{eq:alphakrecursion}). However, the evaluation of $F_k$ relies on solving the aforementioned combinatorial problems associated with fixing the multilink-cluster and the total number of links on its edges
such that the difference between this number of links and the number of edges is at most $k$. 
Thus, it becomes increasingly time-consuming to determine $F_k$ as $k$ increases
%
%can be determined by solving the aforementioned combinatorial problem for a finite number of cases 
and we have implemented this computation up to $K=5$. In particular, we find that $\alpha_0$ and $\alpha_1$ coincide with the result of \cite{BU-18-1}. Interestingly, the polynomials that we found exhibit an intriguing property: they are convex functions of $u$, and writing $\alpha_k$ with respect to the basis of Bernstein polynomials of degree $2k$, i.e.
	\begin{align} \label{eq:alphakrep}
	\alpha_k(u) = \sum_{j=0}^{2k} \alpha_{k,j} %B_{j,2(k-1)}(u)$, where $B_{j,2(k-1)}(u) = {{2(k-1)}\choose{j}} u^j(1-u)^{2(k-1)-j}$
		{{2k}\choose{j}} u^j(1-u)^{2k-j},
	\end{align}
	their coefficients satisfy $0<\alpha_{k,j} \leq 1$ for all $j$ and all $k \leq 5$, see Table \ref{tab:alphak}. 
	Note that, for $k=1$, the occurrence of positive coefficients seems reasonable as $u^j(1-u)^{2-j}$ might account for the contribution of events with $j$ crosses and $2-j$ bars on the sole edge of the multilink-cluster. However, for $k\geq 2$, the combinatorial problems associated with three or more links on one edge of the multilink-cluster need to be included (compare with Table \ref{tab:exploration-scheme-results}), but there is no corresponding basis polynomial. Therefore, this cannot explain this feature and
	hence, we do not know whether this structure persists for larger $k$ or, 
	if it persists, what the reason is.

	\begin{table}	
		\centering
		\resizebox{0.95\textwidth}{!}{$\displaystyle
			\begin{array}[t]{c|cccccc}
			\alpha_{k,j} & k=0 & k=1 & k=2 & k=3 & k=4 & k=5 \\ \hline
			j=0 & 1 & 5/6 & 2/3 & 1559/2520 & 7973/12960 & 375181/604800\\
			j=1 & & 1/2 & 47/120 & 1451/3780 & 71693/181440 & 120203/297000 \\
			j=2 & & 1 & 28/45 & 6737/12600 & 621463/1270080 & 418041641/898128000 \\
			j=3 & & & 1/3 & 353/1260 & 46727/169344 & 70171259/239500800 \\
			j=4 & & & 11/12 & 1721/2700 & 4531/7938 & 122779529/232848000 \\
			j=5 & & & & 9/40 & 210167/1270080 & 122840869/838252800 \\
			j=6 & & & & 307/360 & 226769/317520 & 238710041/349272000 \\
			j=7 & & & & & 57/320 & 8806229/399168000 \\
			j=8 & & & & & 939/1120 & 28680241/35925120 \\
			j=9 & & & & & & 4541/28800 \\
			j=10 & & & & & & 62417/72576
			\end{array}$
		}
		\caption{Coefficients $(\alpha_{k,j})_{j=0}^{2k}$ of the polynomial $\alpha_k$ with respect to the Bernstein basis polynomials %$B_{j,2(k-1)}$ 
			of degree $2k$ for $k=0,\ldots,5$, compare (\ref{eq:alphakrep}).}
		\label{tab:alphak}
	\end{table}

	In addition to the given asymptotic expansion, we may evaluate the aforementioned implicit conditions from Lemma \ref{lem:EM1-estimates} with sufficiently high numerical precision at suitable approximations for $\beta_c$ and, for instance, we find that
	\begin{align} \label{eq:betacestimationexplicit}
		0 \leq \beta_c - \sum_{k=1}^3 \frac{\alpha_{k-1}(u)}{d^k} \leq \frac{2}{d^4}
	\end{align}
	for all $3 \leq d \leq 100$ and $u \in [0,1]$.

\section{The exploration scheme}
\label{sec:exploration-scheme}

The core object we will be working with is an exploration scheme, i.e., a map that assigns a sequence $(M_n)_n$ with $M_n \subseteq V$ to each link configuration $X$. By construction, this process follows the propagation of the loop $\gamma_T$ through the tree and, in particular, the survival of $(M_n)_n$ is related to the event that $|\gamma_T|=\infty$. Moreover, every $x \in M_n$ will have an ancestor within $M_{n-1}$ which is not necessarily the predecessor of $x$. Rather, given its ancestor, $x$ is chosen in a way such that the edge preceding %predecessing 
$x$ carries one link that renews $\gamma_T$ in a certain way. From this renewal property and for $X$ given by Poisson point processes, it follows that $(|M_n|)_n$ is a Galton-Watson process and we may therefore characterise its survival probability by $\mathbb E(|M_1|)$. Fortunately, we can calculate this expected value quite well, resulting in both theorems from Section \ref{sec:results}.

Before we may get into a detailed analysis, let us fix some
notation. For $x,y \in V$, we write $x \sim y$ iff $\{x,y\} \in E$ and $y \geq x$ iff the unique 
shortest path from $y$ to the root contains $x$. 
A connected subgraph $S$ of $T$ with $x \in V(S)$ and $y \geq x$ for 
all $y \in S$ is called a subtree of $T$ with root $x$. 
Given such a subtree $S$ of 
$T$ with root $x$, we write $S^+$ for the enlargement of
$S$ by one layer, i.e., $S^+$ is the subtree with edge set 
$E(S^+) = \{ e \in E: e \cap V(S) \neq \emptyset \text{ and }e \neq e_x^-\}$, and with 
vertex set $V(S^+) = \{ x \in V: x \in e \text{ for some } e \in 
E(S^+)\}$.
%We write $\partial_V(S^+)=V(S^+)\setminus V(S)$ for brevity. 
Here, for $x \neq r$, $e_x^-=\{\pred(x),x\}$ denotes the %unique 
edge %adjacent to 
from $x$ to its predecessor $\pred(x)$. % whose removal from $T$ would disconnect $x$ from the root $r$ (in the tree) is denoted by $e_x^-$. 
Moreover, for a subgraph $S \subseteq T$ and a link configuration $X$ on $T$, we obtain the link configuration $X_S$ by retaining only the links on edges of $S$. 
Additionally, if $S \subseteq T$ is a subtree, $x \in V(S)$ and $t \in \Tbeta$, we write $\gamma_{S,x,t}$ for the loop induced by $X_S$ on 
$S$ that contains $(x,t)$. In particular, if $r \in V(S)$, %is a subtree of $T$ containing the root $r$, 
we write $\gamma_S=\gamma_{S,r,0}$ for brevity. 
Finally, we write 
$N^e := X^{e,\cross}(\Tbeta)+X^{e,\dbars}(\Tbeta)$ for the total number of links on an edge $e \in E$.

The basic observation that our method is based on is the following renewal property.

\medskip
\begin{lem} \label{lem:loop-splitting}
	Let $\{x,y\} \in E$ with $N^{\{x,y\}}=1$, i.e. $\supp( X^{\{x,y\},\cross}+X^{\{x,y\},\dbars}) = \{t\}$ for some $t \in \Tbeta$.  Denote by $S_x$ and $S_y$ the distinct subtrees of $T$ such that $x \in V(S_x)$, $y \in V(S_y)$,  $V(S_x) \cup V(S_y)=V$ and $\{x,y\} \notin E(S_x) \cup E(S_y)$. %vertices that are closer to $x$ and $y$, respectively. 
	Then for any loop $\gamma$ that crosses $\{x,y\}$, i.e. such that $\gamma \cap \{x\} \times U \neq \emptyset$ for any open neighbourhood $U$ of $t$, we have
	\begin{align} \label{eq:loop-splitting}
	\gamma %\cup \{(x,t),(y,t)\} 
	\subseteq \gamma_{S_x,x,t} \cup \gamma_{S_y,y,t}.
	\end{align}
	Moreover, if $|\gamma|<\infty$, we even have equality within (\ref{eq:loop-splitting}) except for the points $(x,t)$ and $(y,t)$.
	\begin{proof}
		%		First, we note that connecting points within $V\times \Tbeta$ according to $X$ or $X_{S_x}$ may only be different if the corresponding path crosses the edge $\{x,y\}$. 
		If $(x,t-)\in \gamma$% (or $(y,t+) \in \gamma_T$, respectively)%
		, then we distinguish between two cases:
		\begin{enumerate}[label=(\textit{\arabic*})]
			\item If $(x,t+) \in \gamma$% (or $(y,t-)\in \gamma_T$, respectively)%
			,
%			any path $\Gamma$ connecting two points within $V(S_x)\times \Tbeta$ that arriving at $(x,t-)$ may  
			points in %connecting points within 
			$V(S_x) \times \Tbeta$ are connected according to $X$ if and only if they are connected according to $X_{S_x}$ -- with the exception of the point $(x,t)$.
			\item If $(x,t+) \notin \gamma$% (or $(y,t-)\notin \gamma_T$, respectively)%
			, there is no possibility for the connecting path to come back to $S_x$ as the underlying graph is a tree and the path needs a link to cross from $y$ to $x$. Thus, ignoring the link on $\{x,y\}$ will increase the set of points within $V(S_x)\times \Tbeta$ that are connected.
		\end{enumerate}
		The same argument holds if we initially had $(x,t+) \in \gamma$ (with $t-$ and $t+$ exchanged) and this shows (\ref{eq:loop-splitting}). Moreover, if 
		$\gamma$ is a finite loop, then it is closed, meaning that two points within $\gamma$ are connected by two distinct paths. Thus, since the underlying graph is a tree and in comparison with the link configuration $\tilde X$ one obtains from $X$ by removing the link on $\{x,y\}$, the addition of this link affects at most one of these paths. Therefore, the points $(x,t-)$ and $(x,t+)$ that were connected w.r.t.\ $\tilde X$ remain connected w.r.t. $X$, compare with \cite[Proposition 2.2]{BU-18-1}. This means that the case (\textit{2}) cannot occur for 
		$|\gamma|<\infty$ and we obtain the asserted equality.
	\end{proof}
\end{lem}

Note that, in general, we do not know whether case (\textit{1}) or (\textit{2}) holds by just considering $X_{S_x}$. However, splitting a loop $\gamma(X)$ into $\gamma_{S_x,x,t}(X_{S_x})$ and $\gamma_{S_y,y,t}(X_{S_y})$ gives an upper bound for the propagation of $\gamma$ that is optimal in the sense that at least for $|\gamma|<\infty$ we have equality.

To apply this observation, assume that we are given a link configuration $X$ on $T$. Now, we explore the tree starting from the root and
%, for some $x \in V$, 
consider the multilink-cluster $\bar C_x$ rooted in some $x \in V$. That is, $\bar C_x$ is the maximal subtree with root $x$ such that each of its 
edges has at least two links, i.e. 
\begin{align*}
\bar C_x %= \bar C_x(X) 
:= \bigcup \{ S \subseteq T: S \text{ subtree with root } x, 
N^e \geq 2 \text{ for all } e \in E(S) \}.
\end{align*}
If this subtree is infinite, we may not be able to apply Lemma \ref{lem:loop-splitting} to divide the propagation of $\gamma_T$ into finite segments, therefore we set
\begin{align*}
C_x := 
\begin{cases}
\bar C_x & \text{ if } |\bar C_x| < \infty,\\
\emptyset & \text{ otherwise.} 
\end{cases}
\end{align*}

The exploration scheme 
is then defined recursively by
$M_0 := \{r\}$ and 
\begin{align*}
M_{n+1} &:= \bigcup_{x \in M_n} M_1^x, \qquad n=0,1,2,\ldots
\intertext{with}
M_1^x &:=\{y \in V(C_x^+)\setminus V(C_x) : y \in \gamma_{C_x^+,x,t_x}\},
\end{align*}
where $t_r=0$ and $\{t_x\}=\{t_x(X)\}:=\supp X^{e_x^-}$ for $x \neq r$. A sketch of these quantities is given in Figure \ref{fig:exploration-scheme}.

\begin{figure}
	\centering
	\begin{tikzpicture}
		% nodes of the tree
		{%
			\node[circle,scale=0.3,fill=myyellow,color=myyellow,draw] at (0,0) (r) {};
			\node[circle,scale=0.3,fill=myyellow,color=myyellow,draw] at (-3,2) (x1) {};
			\node[circle,scale=0.3,fill=myyellow,color=myyellow,draw] at (-4,4) (y1) {};
			{
				\node[circle,scale=0.3,color=mypurple,draw] at (r) {};
				\node[circle,scale=0.3,color=mypurple,draw] at (x1) {};
				\node[circle,scale=0.3,color=mypurple,draw] at (y1) {};
			}
			\node[circle,scale=0.3,fill=myyellow,color=myyellow,draw] at (0,2) (x2) {};%
		}
		\node[below] at (r) {{{{$r$}}}};
		\node[circle,scale=0.3,fill=black,color=black,draw] at (3,2) (x3) {};
		\node[circle,scale=0.3,fill=black,color=black,draw] at (-3,4) (y2) {};
		\node[circle,scale=0.3,fill=black,color=black,draw] at (-1,4) (y4) {};
		\node[circle,scale=0.3,fill=black,color=black,draw] at (0,4) (y5) {};
		\node[circle,scale=0.3,fill=black,color=black,draw] at (1,4) (y6) {};
		\node[circle,scale=0.3,fill=black,color=black,draw] at (2,4) (y7) {};
		\node[circle,scale=0.3,fill=black,color=black,draw] at (3,4) (y8) {};
		\node[circle,scale=0.3,fill=black,color=black,draw] at (4,4) (y9) {};
		\node[circle,scale=0.3,fill=black,color=black,draw] at (-4.3,6) (z1) {};
		\node[circle,scale=0.3,fill=black,color=black,draw] at (-4,6) (z2) {};
		\node[circle,scale=0.3,fill=black,color=black,draw] at (-3.3,6) (z4) {};
		\node[circle,scale=0.3,fill=black,color=black,draw] at (-3,6) (z5) {};
		\node[circle,scale=0.3,fill=black,color=black,draw] at (-2.7,6) (z6) {};
		\node[circle,scale=0.3,fill=black,color=black,draw] at (-2.3,6) (z7) {};
		\node[circle,scale=0.3,fill=black,color=black,draw] at (-2,6) (z8) {};
		\node[circle,scale=0.3,fill=black,color=black,draw] at (-1.7,6) (z9) {};
		\node[circle,scale=0.3,fill=black,color=black,draw] at (-1.3,6) (z10) {};
		\node[circle,scale=0.3,fill=black,color=black,draw] at (-1,6) (z11) {};
		\node[circle,scale=0.3,fill=black,color=black,draw] at (-0.7,6) (z12) {};
		{
			\node[circle,scale=0.3,fill=myred,color=myred,draw] at (-3.7,6) (z3) {};
			\node[circle,scale=0.3,fill=myred,color=myred,draw] at (-2,4) (y3) {};
		}
		{
			\node[circle,scale=0.3,color=mypurple,draw] at (z3) {};
			\node[circle,scale=0.3,color=mypurple,draw] at (y3) {};			
			\node[circle,scale=0.3,color=mypurple,draw] at (z8) {};
		}
		\node at (-2.5,7) (dots1) {$\vdots$};
		\node at (1.5,5) (dots2) {\rotatebox{90}{$\ddots$}};
		{%
			{%
				\draw[color=myyellow] (r) -- (x1) -- (y1);
				\draw[dashed,color=mypurple] (r) -- node {{\color{myblue}{$3$}}}%node {\visible<2->{\color{myblue}{$3$}}} 
				(x1) -- node {{\color{myblue}{$2$}}} (y1);%
			}
			\draw[color=myyellow] (r) -- node {{\color{myblue}{$2$}}} (x2);%
		}
		\draw[dotted] (y1) -- node {{\color{myblue}{$1$}}} (z1);
		\draw[dotted] (z4) -- (y2) -- node {{\color{myblue}{$0$}}} (x1);
		\draw[dotted] (x2) -- node {{\color{myblue}{$1$}}} (y4) -- (z10);
		\draw[dotted] (z2) -- node {{\color{myblue}{$0$}}} (y1);
		{
			\draw[dotted] (y1) -- (z3);
			\draw[dashed,color=mypurple] (y1) -- node {{\color{myblue}{$1$}}} (z3);
			\draw[dotted] (x1) -- (y3) -- (z7);
			\draw[dashed,color=mypurple] (x1) -- node {{\color{myblue}{$1$}}} (y3);
		}
		\draw[dotted] (z5) -- (y2) -- (z6);
		\draw[dotted] (z8) -- (y3) -- (z9);
		\draw[dotted] (z11) -- (y4) -- (z12);
		\draw[dotted] (y5) -- node {{\color{myblue}{$1$}}} (x2) -- node {{\color{myblue}{$1$}}} (y6);
		\draw[dotted] (r) -- node {{\color{myblue}{$0$}}} (x3) -- (y7);
		\draw[dotted] (y8) -- (x3) -- (y9);
	\end{tikzpicture}
	\caption{On the $3$-ary tree, the numbers {\color{myblue}{$N^e$}} of links on the edges $e \in E$ constitute the multilink-cluster {\color{myyellow}{$C_r$}}. Whereever the loop {\color{mypurple}{$\gamma_T$}} crosses an edge protruding from {\color{myyellow}{$C_r$}}, a vertex of {\color{myred}{$M_1$}} occurs.}
	\label{fig:exploration-scheme}
\end{figure}
Note that, for $n \in \mathbb N$, there is always exactly one link on the edge $e_x^-$ preceding any $x \in M_n$. Therefore, 
we may indeed apply Lemma \ref{lem:loop-splitting} to these edges and obtain that -- if $\gamma_T$ is finite -- the trace of $\gamma_T$ within $C_x^+$ coincides with $\gamma_{C_x^+,x,t_x}$. Thus, $\gamma_T$ reaches the boundary vertices $V(C_x^+)\setminus V(C_x)$ of $C_x^+$ if and only if the loop $\gamma_{C_x^+,x,t_x}$ does so and the latter information is encoded within $M_1^x$. On the other hand, if $\gamma_T$ is infinite, then $\gamma_{C_x^+,x,t_x}$ is an upper bound for the trace of $\gamma_T$ within $C_x^+$. This allows us to relate the survival/extinction of $(M_n)_n$ to the infiniteness/finiteness of $\gamma_T$.

\medskip
\begin{prop}\label{prop:exploration-scheme-conditions} Fix a link configuration $X$.
	\begin{enumerate}[label=\textnormal{(\textit{\alph*})}]
		\item If $\left| \bigcup_{n \in \mathbb N_0} M_n \right| %\sum_{n \in \mathbb N_0} |M_n|
		=\infty$, then $|\gamma_T|=\infty$.
		\item If $\left| \bigcup_{n \in \mathbb N_0} M_n \right| % $\sum_{n \in \mathbb N_0} |M_n|
		<\infty$ and $|\bar C_x|<\infty$ for all $x \in \bigcup_{n \in \mathbb N_0} M_n$, then $|\gamma_T|<\infty$.
	\end{enumerate}	
\end{prop}
\begin{proof} 
Let us begin with two observations that hold for any $x \in \bigcup_n M_n$. On the one hand, for $y \in V(C_x^+)\setminus V(C_x)$ with $N^{e_y^-}=1$, we have $y \in M_1^x$ iff $(\pred(y),t_y) \in \gamma_{C_x,x,t_x}$%, where $\pred(y)$ denotes the predecessor of $y$
. On the other hand, we may apply Lemma \ref{lem:loop-splitting} to edges $\{e_x^-\}\cup \big(E(C_x^+) \setminus E(C_x)\big)$ to find
\begin{align} \label{eq:gammaTrestricted}
%\gamma_T \big |_{C_x} := 
\gamma_T \, \cap \, V(C_x) \times \Tbeta \subseteq \gamma_{C_x,x,t_x}.
\end{align} 	
Moreover, for $|\gamma_T|<\infty$ we even have equality within (\ref{eq:gammaTrestricted}) up to finitely many isolated points.\\
Now, to prove (\textit{a}), suppose that $\bigcup_{n \in \mathbb N_0} M_n$ is infinite and pick a sequence $(x_n)_{n \in \mathbb N_0}$ with $r=x_0 \leq x_1 \leq \ldots$ as well as $x_n \in M_n$ for all $n$. If we further assume that $|\gamma_T|<\infty$, we may set $n_0 :=\max \{n : x_n \in \gamma_T\}$, $x:=x_{n_0}$ and $y := x_{n_0+1}\in M_1^x$. Combining the two observations from the beginning of this proof, this gives $y \in \gamma_T$ -- in contradiction to the maximality of $n_0$.\\
For (\textit{b}), suppose that $|\gamma_T|=\infty$ and choose $(x_k)_{k \in \mathbb N_0} \subseteq V$ with $r=x_0 \leq x_1 \leq \ldots$, $x_0 \sim x_1 \sim \ldots$ as well as $x_k \in \gamma_T$ for all $k$. If, however, $\bigcup_n M_n$ is finite, then there is $k_0:=\max \{k : x_k \in \bigcup_{n \in \mathbb N_0} M_n\}$%, in particular $M_1^x = \emptyset$ for 
. Thus, we may set $x:=x_{k_0}$ and $y:=x_{k_1+1}$, where $k_1:=\max \{k \geq k_0 : x_k \in \bar C_x\}$ is finite by assumption. By the second preliminary observation, we find $(\pred(y),t_y) \in \gamma_{C_x,x,t_x}$ and thus $y \in M_1^x$ in contradiction to maximality of $k_0$.
\end{proof}

Now, let us assume that we are given link configurations at random. As mentioned before, for each realisation we 
may trace the (possible) propagation of $\gamma_T$ within the finite segments $C_x$ for some $x \in M_n$, $n \in \mathbb N_0$, by considering the loop $\gamma_{C_x,x,t_x}$,
and the random variables $M_1^x$ keep track where to start with new segments. Since this only relies on \emph{local} information about $X$, it is no surprise that $(|M_n|)_n$ forms a Galton-Watson process under natural conditions on the distribution of $X$.

\medskip
\begin{lem} \label{lem:Mn-GW}
	Let $(X^{e,\star})_{e \in E, \star \in \{ {\cross},{\dbars}\}}$
	be a family of admissible 
	point processes on $\Tbeta$. Assume that 
	the family $(X^{e,{\cross}}, X^{e,{\dbars}})_{e \in E}$ is 
	independent and identically distributed, and that each 
	$X^{e,\star}$ is invariant under shifts in $\Tbeta$. 
	Then %$\mathbb P(\left| 
	$\bigcup_{n \in \mathbb N_0} M_n $ %\right| < \infty)=1$
	is infinite with positive probability if and only if $\mathbb E(|M_1|)>1$.
\end{lem}
\begin{proof}
	To begin with, we have
	\begin{align*}
	\varphi_{n+1}(w) = \mathbb E\left(w^{|M_{n+1}|}\right) = \sum_{\Pi} \mathbb E \left( w^{|M_{n+1}|} \mathbf{1}_{\{M_n=\Pi\}} \right),
	\end{align*}
	for $w \in [0,1]$, where $\varphi_n$ denotes the probability generating function of $|M_n|$ and where the sum runs over all subsets $\Pi$ of the leaves of some finite subtree of $T$. Now fix $\Pi$, let $S_x$ be the subtree of $T$ with root $x \in \Pi$ and set $S_\Pi:=T\setminus \bigcup_{x \in \Pi}S_x$ to be tree containing all remaining edges. Furthermore, for a realisation of $X$ within $\{M_n=\Pi\}$ we identify $X$ with $(X_{S_\Pi},(X_{S_x})_{x \in \Pi})$, where $X_S$ represents the links on edges $e \in E(S)$. Then, by definition, we have
	\begin{align*}
	|M^x_1(X)| = |M_1 (\Theta_{x,t_x(X_{S_\Pi})}(X_{S_x}))|
	\end{align*}
	for $x \in \Pi$. Here, $\Theta_{x,t}$ takes the links of $X_{S_x}$ and applies a position shift by $t$ as well as a spatial shift by some tree-isomorphism from $S_x$ to $T$ to these links. Since the first of these shifts leaves the distribution of $X_{S_x}$ invariant and the second maps it to the distribution of $X$, Fubini's theorem implies
	\begin{align*}
	\varphi_{n+1}(w) %\sum_{\Pi}\mathbb E\left(w^{|M_{n+1}|}\mathbf{1}_{\{M_n=\Pi\}} \right) 
	&= \sum_{\Pi}\mathbb E \left( \prod_{x \in \Pi} w^{|M_1^x|} \mathbf{1}_{\{M_n=\Pi\}} \right) \\
% ************** Do not remove **************
%	&= \sum_{\Pi} \int \mathbb P \circ X_{S_\Pi}^{-1}(\mathrm{d}\omega_{S_\Pi}) \mathbf{1}_{\{M_n=\Pi\}}(\omega_{S_\Pi}) \\
%	& \qquad \quad \prod_{x \in \Pi} \int \mathbb P \circ X_{S_x}^{-1} (\mathrm{d}\omega_{S_x}) w^{|M_1 \circ \Theta_{x,t_x(\omega_{S_\Pi})}(\omega_{S_x})|}\\
%	&= \sum_{\Pi} \int \mathbb P \circ X_{S_\Pi}^{-1}(\mathrm{d}\omega_{S_\Pi}) \mathbf{1}_{\{M_n=\Pi\}}(\omega_{S_\Pi}) \\
%	& \qquad \quad \prod_{x\in \Pi} \int \underbrace{\mathbb P \circ X_{S_x}^{-1} \circ \Theta_{x,t_x(\omega_{S_\Pi})}^{-1}}_{\mathbb P \circ X^{-1}} (\mathrm{d}\omega_{S_x}) w^{|M_1(\omega_{S_x})|}\\
% ************** **************
	&= \sum_{\Pi} \int \mathbb P(\mathrm{d}X_{S_\Pi}) \mathbf{1}_{\{M_n=\Pi\}}(X_{S_\Pi}) \\
	& \qquad \quad \prod_{x \in \Pi} \int \mathbb P (\mathrm{d}X_{S_x}) w^{|M_1 \circ \Theta_{x,t_x(X_{S_\Pi})}(X_{S_x})|}\\
	&= \sum_{\Pi} \mathbb E \left( \mathbb E(w^{|M_1|})^{|\Pi|} \mathbf{1}_{\{M_n=\Pi\}} \right) \\
	&= \mathbb E\left(\varphi_1(w)^{|M_n|}\right) = \varphi_n\circ \varphi_1(w).
	\end{align*}
	Thus, by $\mathbb P(|M_1|=1)<1$ and since $|M_n|=0$ implies $|M_{n+1}|=0$, the standard (fixed-point) argument from the theory of Galton-Watson processes implies the asserted equivalence (compare \cite[chapter I.3 and I.5]{AthNey}).
\end{proof}

Note that -- with a little bit more effort -- we could also show that $(|M_n|)_n$ is a Galton-Watson process. However, the stated characterisation of survival suffices for our purposes. In particular, by
Proposition \ref{prop:exploration-scheme-conditions} and Lemma \ref{lem:Mn-GW} it is clear that we need to be interested
in ${\mathbb E}(|M_1|)$. For concreteness and because this is the most important situation, we only study this quantity 
in the case of the Poisson point processes described in the 
previous section.

For a concise presentation, we set
\begin{align*}
\begin{split}
{\mathcal S}_d := \{ (S,n):& \;S \text{ is a finite subtree of } T \text{ with root }r, \\
& 
\;
n \colon E(S) \to {\mathbb N_0} \text{ with } n(e) \geq 2 \text{ for all } e \in E(S) \}.
\end{split}
\end{align*}
We also write the shorthand $n(S) := \sum_{e \in E(S)} n(e)$, $n!:=\prod_{e\in E(S)}n(e)!$ and define the event
\begin{align*}
A_{S,n}:=\{C_r = S \text{ and }N^e =n(e) \text{ for all }e \in E(S)\}
\end{align*}
for $(S,n) \in \mathcal S_d$. 
By convention, we assume that $(S_0,n_0) \in \mathcal S_d$, where $S_0=(\{r\},\emptyset)$ is the trivial tree
and where $n_0$ is the empty function with $n_0(S_0)=0$.

\medskip
\begin{lem} \label{lem:EM1-rep}
	Let $(X^{e,\star})_{e \in E, \star \in \{{\dbars},{\cross}\}}$ 
	be independent homogeneous Poisson point processes on $\Tbeta$, with 
	rate $u$ for $X^{e, {\cross}}$ and $(1-u)$ for 
	$X^{e, {\dbars}}$. Then there exist nonnegative coefficients 
	$p_{S,n}(d,u)$ (independent of $\beta$ and polynomial in $u$) with 
	\begin{align} \label{eq:EM1rep}
	{\mathbb E}(|M_1|) = \sum_{(S,n) \in {\mathcal S}_d} \left( {\mathrm{e}}^{-\beta d}
	(1+\beta)^{d-1} \right)^{|V(S)|} \beta^{n(S)+1}  p_{S,n}(d,u).
	\end{align}
\end{lem}
For each $(S,n) \in \mathcal S_d$, the polynomials $p_{S,n}(d,u)$ can be calculated: Example \ref{ex:order0} will deal with the most basic case and within Section \ref{sec:combinatorics}, we will see how to reduce this calculation to a combinatorial problem for arbitrary $(S,n)$.
\begin{proof}[Proof of Lemma \ref{lem:EM1-rep}]
	We decompose 
	\[
	{\mathbb E}(|M_1|) = 
	\sum_{(S,n) \in {\mathcal S}_d} {\mathbb P} \big( A_{S,n} \big) 
	{\mathbb E} \big( |M_1| \, \big| \, A_{S,n} \big). 	 
	\]
	By independence, and using also the facts that $|E(S^+) \setminus
	E(S) | = d|V(S)| -|E(S)|$ and $|E(S)| = |V(S)| - 1$, we find 
	\[
	\begin{split}
	{\mathbb P} \big( A_{S,n} \big) & = \prod_{e \in E(S)}
	{\mathbb P}(N^e = n(e)) \prod_{e \in E(S^+) \setminus E(S)} 
	{\mathbb P}( N^e \leq 1)  \\
	& = \frac{\prod_{e \in E(S)} \beta^{n(E)} 
		}{\prod_{e \in E(S)} n(e)!} 
	{\mathrm{e}}^{- \beta |E(S)|} \big( (1+\beta) {\mathrm{e}}^{-\beta} \big)^{d|V(S)| -|E(S)|}  \\
	& = \frac{\beta^{n(S)}}{n!}  
	{\mathrm{e}}^{- \beta  d |V(S)|} (1+\beta)^{(d-1)|V(S)| + 1}.	
	\end{split}
	\]
	On the other hand, 
	\[
	{\mathbb E} \big( |M_1| \, \big| \, A_{S,n} \big)
	= \sum_{y \in V(S^+) \setminus V(S)} 
	{\mathbb P} \big ( y\in \gamma_{S^+} \, \big| 
	\, A_{S,n} \big), 
	\]
	and 
	the term in the sum on the right hand side above can be 	
	written as 
	\[
	{\mathbb P} \big ( y \in \gamma_{S^+} \, \big| 
	\, A_{S,n}, N^{e_y^-} = 1 \big) 
	{\mathbb P}\big( N^{e_y^-} = 1 \, \big| \,  
	A_{S,n} \big),
	\]
	where we used that $y \in \gamma_{S^+}$ implies $N^{e_y^-}=1$.
	Now, for all $y$, the second factor above is equal to 
	${\mathbb P}( N^{e_y^-} = 1 | N^{e_y^-} \leq 1) = \frac{\beta}{1+\beta}$ 
	by independence. 
	Moreover, the first factor does not depend on $\beta$: By
	\begin{align*}
	\{y \in \gamma_S^+, N^{e_y^-}=1\} \cap A_{S,n} = \{(\pred(y),t_y)\in \gamma_S, N^{e_y^-}=1\} \cap A_{S,n},
	\end{align*}
	we see that the event depends on the link configuration on edges $e \in E(S) \cup \{e_y^-\}$ and for these edges, the total number $N^e$ of links on each $e$ is fixed. 
	By regarding, for each edge $e$,
	the random variables 
	$(X^{e,\star})_{\star \in \{\cross,\dbars\}}$ 
	as the result of first determining the total
	number of links on $e$ by a Poisson random variable with 
	expectation $\beta$, then determining their type by 
	a Bernoulli random variable with success probability $u$, 
	and then determining the position of their link(s) by a uniform
	random variable on %$\Tbeta$ 
	$\{(s_1,\ldots,s_{N^e}) \in \mathbb T_\beta^{N^e} : s_1\leq \ldots\leq s_{N^e}\}$,
	one sees that ${\mathbb P} ( (\pred(y),t_y) \in \gamma_{S} \, | 
	\, A_{S,n}, N^{e_y^-} = 1 )$ is independent of $\beta$ and polynomial in $u$. Therefore, the claim follows when we put 
	\begin{align}\label{eq:pSn-def}
	p_{S,n}(d,u) := \frac{1}{n!} %{\prod_{e \in E(S)} n(e)!}
	\sum_{y \in V(S^+) \setminus V(S)} {\mathbb P} \big ( (\pred(y),t_y)\in \gamma_{S} \, \big| \, A_{S,n}, N^{e_y^-} = 1 \big).
	\end{align}
\end{proof}

\begin{ex}[Pattern of order $0$] \label{ex:order0}
	The simplest case for $(S,n) \in \mathcal S_d$ is $(S_0,n_0)$ with $S_0=(\{r\},\emptyset)$ being the trivial tree% and $n_0(S_0)=0$
	. Then we have 
	\begin{align*}
	p_{S_0,n_0}(d,u) = \sum_{y \sim r} \underbrace{\mathbb P \left((r,t_y) \in \gamma_{S_0} \, \big| 
		\, C_r=S_0, N^{e_y^-}=1\right)}_{=1\text{ since }\gamma_{S_0}=\{r\}\times \Tbeta} =d.
	\end{align*}
	Note that this is constant in $u$ 
	due to the fact that we do not place any link onto $E(S_0)$ and therefore we don't
	need to distinguish between different types of links. 
\end{ex}

We shall now restrict our attention even further, namely to the case 
$\beta \leq d^{-1/2}$. 
In this case, 
\[
{\mathbb P}(N^e \geq 2) \leq 1 - {\mathrm{e}}^{-d^{-1/2}} ( 1 + d^{-1/2}) < 1/d,
\]
so the cluster of edges that carry two or more links does not 
percolate on the $d$-ary tree. In particular, we almost surely have $|\bar C_x|<\infty$ for all $x \in \bigcup_{n \in \mathbb N_0} M_n$ and by 
combining the results of this section,
we obtain the
following proposition that contains a large portion of the 
proof of Theorem \ref{thrm:sharp-phase-transition}. 
For clarity, we denote the dependence 
of quantities on $\beta$ explicitly below. 

\medskip
\begin{prop}
\label{prop:betac-characterisation}
	The map $\beta \mapsto {\mathbb E}_\beta(|M_1|)$ is strictly increasing and continuous
	on $(0,d^{-1/2}%/\sqrt{d}
	]$. Moreover, the following statements are 
	equivalent:
	\begin{enumerate}[label=\textnormal{(\textit{\alph*})}]
	\item There is a unique and sharp phase transition within $(0,d^{-1/2})$, i.e. there exists a unique $\beta_c \in (0,d^{-1/2})$ such that 
	$\mathbb P_\beta(|\gamma_T| < \infty)=1$ for $\beta \in (0,d^{-1/2})$ if and only if 
	$\beta \leq \beta_c$.
	\item ${\mathbb E}_{\beta=d^{-1/2}}(|M_1|) > 1$.
	\end{enumerate}
	If one (then both) of the above statements holds, then 
	$\beta_c$ is the unique solution to the equation 
	${\mathbb E}_\beta(|M_1|) = 1$, $\beta \in (0,d^{-1/2})$.

	\begin{proof}
		Writing $f_{S,n}(\beta) = \left( {\mathrm{e}}^{-\beta d}
		(1+\beta)^{d-1} \right)^{|V(S)|} \beta^{n(S)+1} p_{S,n}(d,u)$ for the summands within (\ref{eq:EM1rep}) and with $|V(S)|=|E(S)|+1$, we compute 
		\begin{align*}
		\partial_\beta \ln f_{S,n}(\beta) %= (n(S) + 1)/\beta - d |V(S) + (d-1) |V(S)| / (1+ \beta).
		&= \frac{1}{\beta} \left( n(S)-|E(S)|+|V(S)|\frac{1-\beta^2 d}{1+\beta} \right).
		\end{align*}
		Since $n(S)- |E(S)| \geq 2|E(S)|-|E(S)|\geq 0$, this implies 
		\begin{align*}
		\partial_\beta \ln f_{S,n}(\beta) \geq \frac{1}{\beta} |V(S)| 
		\frac{ 1 - \beta d^2}{1+\beta} > 0
		\end{align*}
		whenever $\beta < d^{-1/2}$. By Lemma \ref{lem:EM1-rep}, this shows strict 
		monotonicity. 		
		A direct consequence is that for any finite subset 
		$\hat{\mathcal S}_d$ of $\mathcal S_d$ we find 
		\begin{align*}
		\sup_{\beta \in (0,d^{-1/2} %1/\sqrt{d}
			]} 
		\Big|\mathbb E_\beta(|M_1|) - \sum_{(S,n) \in \hat{ 
		\mathcal S}_d} f_{S,n}(\beta) \Big| = 
		\sum_{(S,n) \notin \hat 
		{\mathcal S}_d} f_{S,n}(d^{-1/2}).
		\end{align*}
		Furthermore, for $\beta = d^{-1/2}$, the expected size of the percolation cluster $\bar C_r$ is finite and thus, $\mathbb E_{\beta = d^{-1/2}}(|M_1|) \leq d \, \mathbb E_{\beta =d^{-1/2}}(|V(\bar C_1)|) < \infty$.	 		 		
		This shows that 
		the series $\sum_{(S,n) \in \mathcal S_d} f_{S,n}(\cdot)$
		of continuous functions converges uniformly on 
		$[0, d^{-1/2}]$, thus its limit
		$\mathbb E_\beta(|M_1|)$ is continuous. \\
		To show the remaining equivalence, note that 
		$\lim_{\beta \downarrow 0} \mathbb E_\beta(|M_1|)=0$.
		Thus, by continuity and monotonicity, there is at 
		most one solution $\beta_c$ of the
		equation $\mathbb E_\beta(|M_1|)=1$ in the interval 
		$(0,d^{-1/2})$, and a necessary and sufficient 
		condition for the existence of such a solution is 
		$\mathbb E_{\beta = d^{-1/2}}(|M_1|) > 1$. 
		Moreover, in this case monotonicity
		implies 
		$\mathbb E_{\beta}(|M_1|) > 1$ for all 
		$\beta \in (\beta_c, d^{-1/2}]$ and $\mathbb E_\beta(|M_1|)\leq 1$ for $\beta \in (0,\beta_c]$. 
		Finally, by Lemma \ref{lem:Mn-GW} and Proposition \ref{prop:exploration-scheme-conditions}, the result follows.
	\end{proof}
\end{prop}

\begin{proof}[Proof of Theorem \ref{thrm:sharp-phase-transition} for $d \geq 5$]
	For the case $d \geq 5$, it is sufficient to estimate $\mathbb E_\beta(|M_1|)$ by the term within (\ref{eq:EM1rep}) that corresponds the trivial tree $(S,n)=(S_0,n_0)$, i.e. $|V(S_0)|=1$ and $n_0(S_0)=0$. Together with Example \ref{ex:order0}, this yields
	\begin{align*}
	\mathbb E_{\beta=d^{-1/2}}(|M_1|) &\geq \mathrm{e}^{-d^{-1/2} d} (1+d^{-1/2})^{d-1} 
	d^{-1/2} d.
	\end{align*}
	For $d \geq 5$, the latter expression is strictly larger than $1$.
%	
%	************** Remark, do not remove! **************
%
%	Taking the formal logarithmic derivative w.r.t.\ $D=\sqrt{d}$ and estimating\linebreak{} $\ln(1+x)\geq \tfrac{x}{x+1}$, we see that the latter expression is increasing in $d$.
%	More precisely, we estimate
%	\begin{align}
%	& \, \frac{1}{D} + 2 D \underbrace{\ln\left(1+\tfrac{1}{D}\right)}_{\geq \frac{1/D}{1/D+1}} + (D^2-1)\frac{1}{1+\tfrac{1}{D}} \frac{-1}{D^2} - 1 \\
%	\geq & \, \frac{1}{\tfrac{1}{D}+1} \left( \frac{1}{D^2} + \frac{1}{D} + 2 - 1 + \frac{1}{D^2} -\frac{1}{D}-1 \right) >0.
%	\end{align}
%	Therefore, (\ref{eq:EM1greater1-1}) is monotone increasing in $d$.
%	This yields
%	\begin{align}
%	\mathbb E_{d^{-1/2}}[|M_1|] \geq \sqrt{5} (1+5^{-1/2})^4 \mathrm{e}^{-\sqrt{5}}  \approx 1.048>1
%	\end{align}
%	and
%	
%	****************************************************
%
	Thus,
	by Proposition \ref{prop:betac-characterisation}, this establishes the existence of a sharp phase transition and the partition into the two phases up to $\betaplus=
	d^{-1/2}$. 
%	Finally, \cite[Proposition 1.2 (2),(4)]{Ham-18} shows that $|\gamma_T|=\infty$ occurs with positive probability 
%	for all $\beta\geq 4 d^{-1}$ and $d\geq 16$. Since $4 d^{-1}\leq d^{-1/2}$ for these values of $d$, this implies that $\mathbb P(|\gamma_T|=\infty)>0$ for all $\beta>\beta_c$.
\end{proof}

To establish the existence of a sharp phase transition for $d=3,4$, too, we need to find sharper estimates on $\mathbb E_{\beta=d^{-1/2}}(|M_1|)$. Thus, we will need to calculate $p_{S,n}$ for more pairs $(S,n) \in \mathcal S_d$. We will do this 
in Section \ref{sec:combinatorics} and these considerations will also enable us to calculate the coefficients $\alpha_k$ within the asymptotic expansion of $\beta_c$.

\section{Asymptotic expansion} \label{sec:asymptotic-expansion} 
In this section, we will prove Theorem \ref{prop:betac-asymptotics}. Since $\beta_c$ is the solution of $\mathbb E_\beta(|M_1|)=1$ (see Proposition \ref{prop:betac-characterisation}), we are going to analyse the representation of $\mathbb E(|M_1|)$
from Lemma \ref{lem:EM1-rep}.
In particular, we are interested in sufficiently precise estimates of $\mathbb E(|M_1|)$ that will be given in Lemma \ref{lem:EM1-estimates}. Apart from providing the tools to establish the asymptotic expansion of $\beta_c$, this lemma will additionally allow us to formulate implicit conditions on $(\beta,d,u)$ such that $\gamma_T$ is finite
almost surely and infinite with positive probability, respectively.

To begin with, let us consider the conditional probabilities 
within the definition (\ref{eq:pSn-def}) of
$p_{S,n}(d,u)$ and note
that, for $y \in V(S^+) \setminus V(S)$ and given $A_{S,n}$ as well as $N^{e_y^-}=1$, the position $t_y$ of the link on $e_y^-$ is independent of $X_S$ and distributed uniformly on $\Tbeta$. Therefore, the conditional probability for $(\pred(y),t_y)$ to be contained in $\gamma_S$ is given by $\mathbb E \big( \tau_S^{\pred(y)}/\beta \,\big|\, A_{S,n} \big)$, where $\tau_S^x$
denotes the time that $\gamma_S$ spends at a vertex $x \in V(S)$, i.e.
\begin{align*}
\tau_S^x = \operatorname{vol}\{t \in \Tbeta : (x,t) \in \gamma_S\}.
\end{align*}
This yields that
\begin{align*}
p_{S,n}(d,u) =& \frac{1}{n!} \sum_{x \in V(S)} (d-d_S^x) \; \mathbb E \left( \frac{\tau_S^{x}}{\beta} \, \bigg| 
\, A_{S,n} \right), %\label{eq:pSn-exp}
\intertext{with}
d_S^x :=& |\{y \in V(S) : \pred(y)=x\}|
\end{align*}
being the out-degree of $x$ within $S$. We will make use of this representation of $p_{S,n}(d,u)$ in Section \ref{sec:combinatorics}. However, for now we will only rely on two observations: On the one hand, $p_{S,n}(d,u)$
is a polynomial in $d$ of degree $1$. On the other hand, $p_{S,n}$ does not change under tree-isomorphisms. This motivates %the following 
to introduce an 
equivalence relation on
$\bigcup_{d \in \mathbb N} \mathcal S_d$ by
\begin{align*}
\begin{split}
& (S,n) \sim (S',n') \\ 
& \Leftrightarrow \text{there is an isomorphism of rooted trees } J \colon S \to S' \text{ such that } %n'=n \circ J^{-1}.
\\& 
\hphantom{
\Leftrightarrow \text{there is an isomorphism of rooted trees } J \colon S \to S' \text{ such that }}
\mathllap{n'=n \circ J^{-1}.}
\end{split}
\end{align*}
To calculate $\mathbb E(|M_1|)$ it then suffices to sum over $\mathcal S := \bigcup_{d \in \mathbb N} \mathcal S_d \big / \sim$ instead of $\mathcal S_d$ if we account for multiplicities
\begin{align*}
\kappa_{S,n}(d) := |[(S,n)] \cap \mathcal S_d|,
\end{align*}
where $[(S,n)]$ denotes the equivalence class of $(S,n)$. 
Some examples of $[(S,n)]$ and the corresponding $\kappa_{S,n}(d)$ are given in
Table \ref{tab:exploration-scheme-results}.
In general, one easily sees that
\begin{align*}
\kappa_{S,n}(d) = \kappa_{S,n}^{(0)} \prod_{\substack{x \in V(S):\\d_S^x \geq 1}} d \cdot (d-1) \cdot \ldots \cdot (d-d_S^x+1)
\end{align*}
with some constant $0< \kappa_{S,n}^{(0)}\leq 1$ that accounts for (in-)distinguishability. In particular, $\kappa_{S,n}$ is a polynomial of degree $\sum_{x \in V} d_S^x = |E(S)|$ and whenever $d < %\max \deg S=
\max \{d_S^x : x \in V(S)\}$, we have $\kappa_{S,n}(d)=0$, consistent with the impossibility of embedding $S$ into the \dreg tree $T$. This allows us to write
\begin{align*}
\mathbb E(|M_1|) =& \sum_{[(S,n)] \in \mathcal S} \left( \mathrm{e}^{-\beta d} (1+\beta)^{d-1} \right)^{|V(S)|} \beta^{n(S)+1} \kappa_{S,n}(d) \, p_{S,n}(d,u). 
\end{align*}
Note that, by introducing $\mathcal S$ and $\kappa_{S,n}(d)$, the index set of summation $\mathcal S$ now does \emph{not} depend on $d$ anymore. This becomes important once we consider the asymptotic behavior of this expression as $d \to \infty$. Furthermore, it turns out to be convenient to introduce the variables $\alpha:=\beta d$ and $h=d^{-1}$, where we may allow arbitrary $h \in \mathbb R$, too. Now, we define the polynomials $q_{S,n}(h,u)$ such that
\begin{align*}
q_{S,n}(d^{-1},u) = d^{-|E(S)|-1} \kappa_{S,n}(d) \, p_{S,n}(d,u)
\end{align*}
for all $d \in \mathbb N$. For $h=d^{-1}$, this immediately gives

\begin{align} \label{eq:EM1-q}
\mathbb E_{\beta=\alpha h}(|M_1|)=& \sum_{[(S,n)] \in \mathcal S} \left( \mathrm{e}^{-\alpha} (1+\alpha h)^{\frac{1}{h}-1} \right)^{|V(S)|} \alpha^{n(S)+1} h^{\ord(S,n)} q_{S,n}(h,u), 
\end{align}

where the \emph{order} of $(S,n)$ is defined by
\begin{align*}
\ord(S,n):=& n(S)-|E(S)|=\sum_{e \in E(S)} (n(e)-1).
\end{align*}
As it turns out, we will need to consider all those terms of (\ref{eq:EM1-q}) with $\ord(S,n) \leq K$ to determine the coefficients $\alpha_0, \ldots,\alpha_K$ from the asymptotic expansion (\ref{eq:betacexpansion}) of $\beta_c$. Therefore, 
for $K \in \mathbb N_0$ and $u \in [0,1]$, we define
\begin{align*}
\FK(\alpha,h,u):= \sum_{\substack{[(S,n)]\in \mathcal S:\\\ord(S,n)\leq K}} \left( \mathrm{e}^{-\alpha}(1+\alpha h)^{\tfrac{1}{h}-1}\right)^{|V(S)|} \alpha^{n(S)+1} h^{\ord(S,n)} q_{S,n}(h,u).
\end{align*}
Note that as $n(S)\geq 2 |E(S)|$ and hence $\ord(S,n) \geq \tfrac{n(S)}{2} \geq |E(S)|$, there are a finite number of equivalence classes $[(S,n)]$ with fixed order $k \in \mathbb N_0$. Thus, $\FK(\cdot,\cdot,u)$ has an analytic continuation onto $\{(\alpha,h) \in \mathbb R^2 : |\alpha h|<1\}$ according to 
\begin{align*}
\mathrm{e}^{- \alpha} (1+\alpha h)^{(\frac 1 h -1)} 
=& \exp \left( \alpha \sum_{k=1}^\infty \tfrac{(-1)^{k}}{k+1} \alpha^k h^{k} - \ln(1+\alpha h) \right). %\label{eq:FK-analytical-expansion}
\end{align*}
This analyticity (in particular for $h=0$) will yield the analyticity of the solution $\alpha^{(K,+)}(h)$ to $F_K(\alpha,h,u)=1$ within the proof of Theorem \ref{prop:betac-asymptotics}.

Finally, we define $\bar{q}_{S,n}$ and $\FKb$ in the same way as $q_{S,n}$ and $\FK$ but with $p_{S,n}$ replaced by
\begin{align*} %\label{eq:pSnbar-def}
\bar p_{S,n} := \frac{1}{n!} \sum_{x \in V(S)} (d-d_S^x) \mathbb E \left( \frac{\beta-\tau_S^x}{\beta} \, \bigg| 
\, A_{S,n} \right).
\end{align*}
Here, $\bar p_{S,n}$ contains the time $\beta-\tau_S^x$ that $\gamma_S$ does \emph{not} spend at a vertex $x\in V(S)$ and in that sense, $\bar p_{S,n}$ is the counterpart of $p_{S,n}$.
Furthermore, note that $\FK$ and $\FKb$ are explicit once we know $\mathbb E\left(\frac{\tau_S^x}{\beta} \, \big| 
\, A_{S,n}\right)$ for all $[(S,n)] \in \mathcal S$ with $\ord(S,n) \leq K$. Within Section \ref{sec:combinatorics}, we will address how to calculate this expected value explicitly.
However, we are now able to state the estimates for
$\mathbb E(|M_1|)$. 
 
\medskip
\begin{lem}[{Estimates of $\mathbb E[|M_1|]$}] \label{lem:EM1-estimates}
	Let $K \in \mathbb N_0$ and $u \in [0,1]$ be arbitrary.
	\begin{enumerate}[label=\textnormal{(\textit{\alph*})}]
 		\item For all $d \in \mathbb N$ and $\beta>0$ we have
 		\begin{align*}
 		\mathbb  E %_{\beta,d,u}
 		(|M_1|) \geq \FK(\beta d, %\tfrac 1d
 		d^{-1}, u).
 		\end{align*}
 		\item For all $d \in \mathbb N$ and $\beta >0$ with $d(1-\mathrm{e}^{-\beta}(1+\beta)) < 1$ we have
 		\begin{align*}
 		\mathbb E %_{\beta,d,u}
 		(|M_1|) \leq \frac{\beta d \mathrm{e}^{-\beta}}{1-d(1-\mathrm{e}^{-\beta}(1+\beta))} - \FKb (\beta d, %\tfrac 1d
 		d^{-1}, u).
 		\end{align*}
 		\item For all $\hat \alpha>\mathrm{e}^{-2}$ and $d_0 \in \mathbb N$ with $d_0 > \hat \alpha^2 \mathrm{e}^2$ there is a constant $c_K>0$ such that for all $d \geq d_0$ and all $0<\alpha\leq \hat \alpha$ we have
 		\begin{align*}
 		\mathbb E_{\beta=\alpha/d}(|M_1|) \leq \FK (\alpha,d^{-1}, u) +
 		\frac{c_K}{d^{K+1}}.
% 		c\left(\frac{\hat \alpha^2 \mathrm{e}^2}{d}\right)^{K+1}.
		\end{align*}
		Moreover, $c_K \leq c \, (\hat \alpha^2 \mathrm{e}^2)^{K+1}$ for some constant $c$.
	\end{enumerate}
\end{lem}

\begin{figure}[h]
	\resizebox{\textwidth}{!}{
	\begin{tabular}{p{0.5\textwidth} p{0.5\textwidth}}
 	\includegraphics[height=165px]{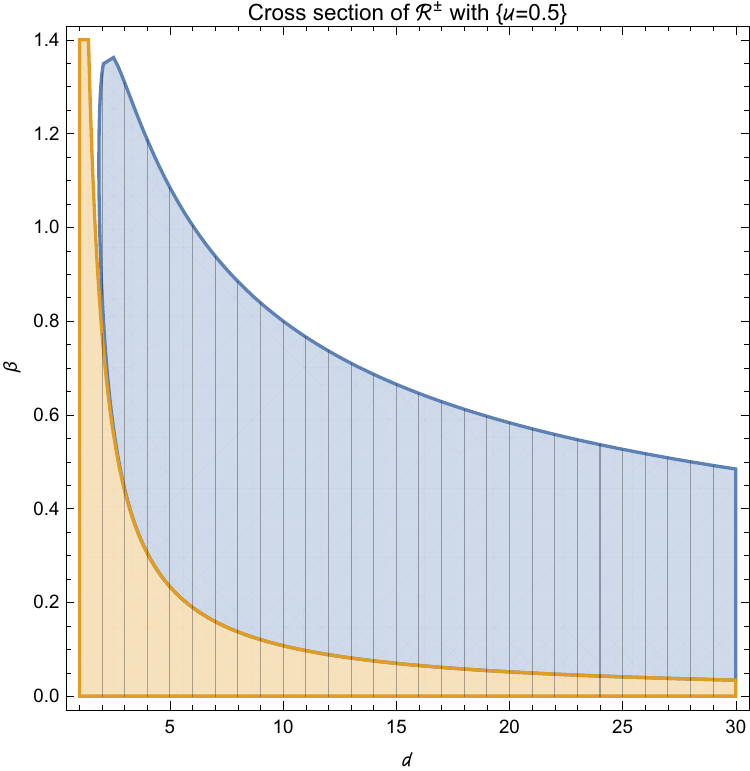} &
 	\includegraphics[height=166px]{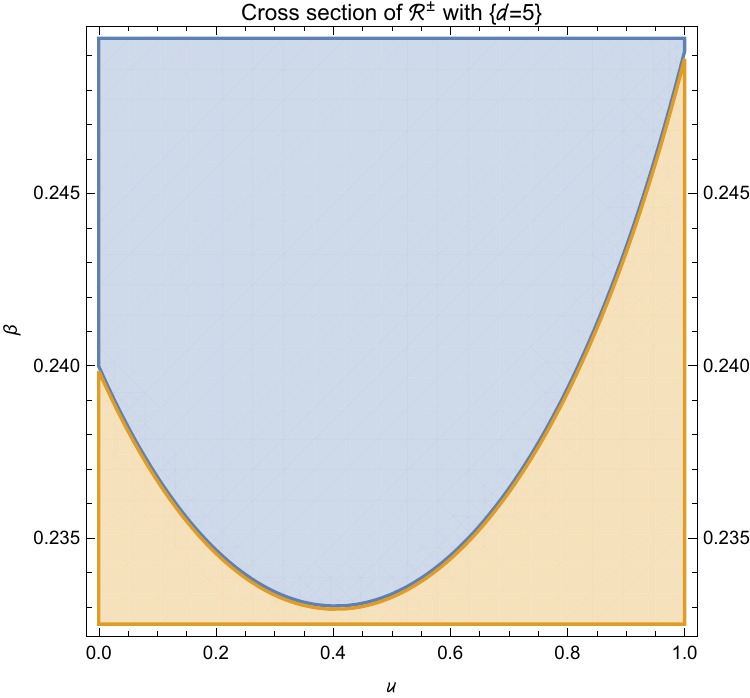}\\
 	\includegraphics[height=165px]{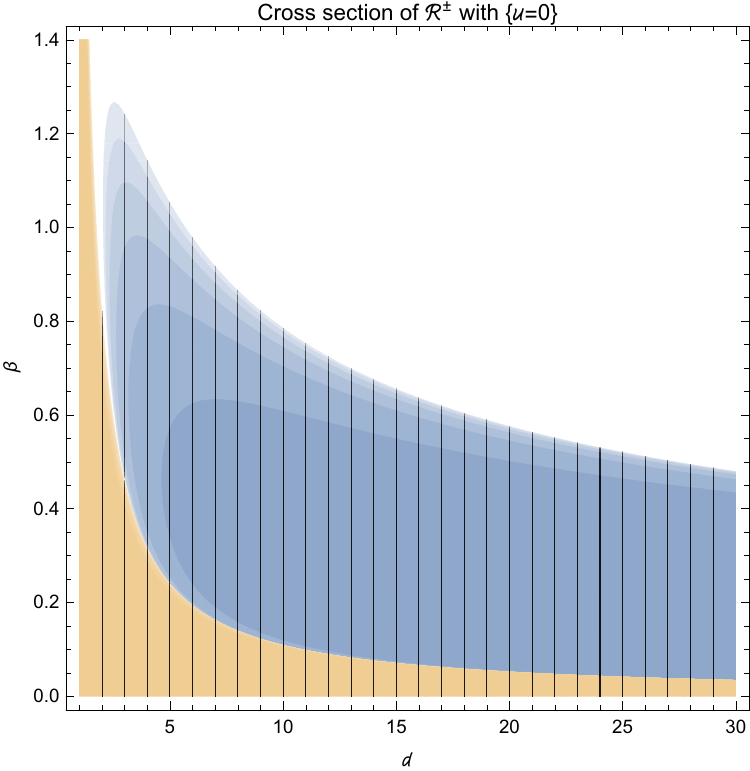} & \hspace{3px}
 	\includegraphics[height=165px]{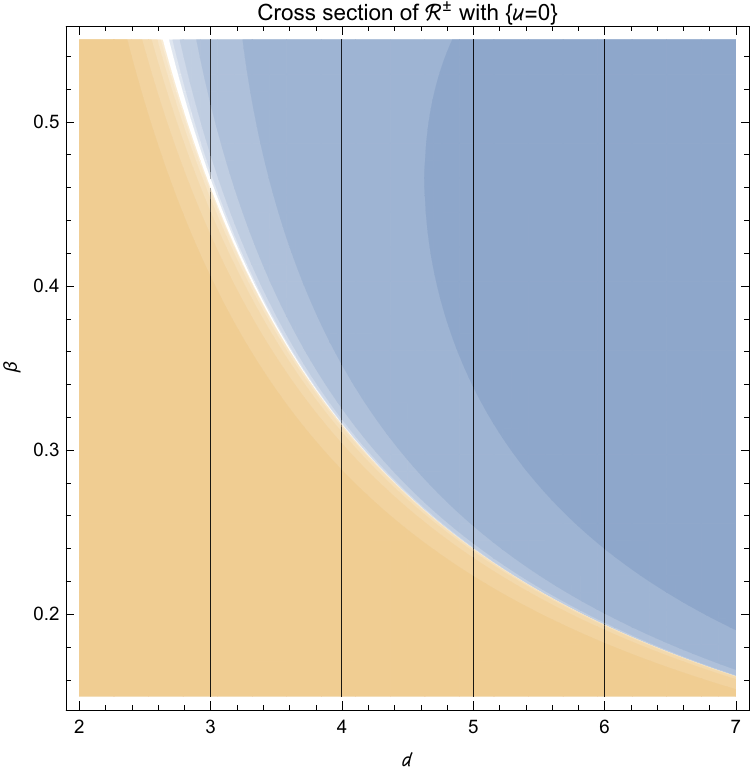}
 	\end{tabular}
 }
 	\caption{Regions $\mathcal R_K^\pm$ of parameters $(\beta,d,u)$ where we can guarantee that $\gamma_T$ is infinite with positive probability (blue region $\mathcal R^+_K$) and that $\gamma_T$ is finite almost surely (sandybrown region $\mathcal R^-_K$), respectively. On top, we considered $K=5$ while the bottom pictures show a comparison for $K=0,\ldots,5$ with regions of higher $K$ being more lightly coloured.}
 	\label{fig:phase-subregions-2}
\end{figure}

Before addressing the proof, let us look at an immediate consequence. If we combine the estimates of Lemma \ref{lem:EM1-estimates}(\textit{a}) and (\textit{b}) with Proposition \ref{prop:exploration-scheme-conditions}
and Lemma \ref{lem:Mn-GW}, we see
that with positive probability there are infinite loops for all parameters within the region
\begin{align}
%(\beta,d,u) \in \mathcal R^+_K&:=\{(\tilde \beta, \tilde d, \tilde u) \in (0,\infty)\times \mathbb N \times [0,1] : \FK ( \tilde \beta \tilde d,  \tfrac{1}{\tilde d}, \tilde u)>1\}, \label{eq:Rplusdef}
\mathcal R^+_K&:=\{( \beta, d, u) \in (0,\infty)\times \mathbb N \times [0,1] : \FK ( \beta d,  d^{-1}, u)>1\}, \label{eq:Rplusdef}
%\end{align}
\intertext{%
	while $\gamma_T$ is finite %$\mathbb P_{\beta,d,u}$-
	almost surely for 
}
%\begin{align}
\begin{split}
%(\beta,d,u) \in \mathcal R^-_K &:=\Big\{ (\tilde \beta, \tilde d, \tilde u) \in (0,\infty)\times \mathbb N \times [0,1] : \tilde d(1-\mathrm{e}^{- \tilde \beta}(1+ \tilde \beta)) \leq 1 \\
%& \qquad \text{ and } \frac{{\tilde \beta} \tilde d \mathrm{e}^{-{\tilde \beta}}}{1-{\tilde d}(1-\mathrm{e}^{-{\tilde \beta}}(1+{\tilde \beta}))} - \FKb ({\tilde \beta} {\tilde d}, \tfrac{1}{\tilde d}, {\tilde u}) \leq 1 \Big\}.
\mathcal R^-_K &:=\Big\{ (\beta, d, u) \in (0,\infty)\times \mathbb N \times [0,1] : d(1-\mathrm{e}^{- \beta}(1+ \beta)) < 1 \\
& \qquad \qquad \text{ and } \frac{{ \beta} d \mathrm{e}^{-{ \beta}}}{1-{ d}(1-\mathrm{e}^{-{\beta}}(1+{\beta}))} - \FKb ({\beta} {d}, d^{-1}, {u}) \leq 1 \Big\}.
\end{split}\label{eq:Rminusdef}
\end{align}

Various cross sections of $\mathcal R^\pm_K$ are shown in Figure \ref{fig:phase-subregions} and Figure \ref{fig:phase-subregions-2}, with the latter figure also containing a comparison of the precision of $\mathcal R^\pm_K$ for $K=0,\ldots,5$. 

\begin{proof}[Proof of Lemma \ref{lem:EM1-estimates}]
	The estimate within (\textit{a}) follows directly from (\ref{eq:EM1-q}) and the definition of $\FK$. Moreover, as
	\begin{align}\label{eq:p+pbar}
	p_{S,n}(d,u) \leq p_{S,n}(d,u) + {\bar p}_{S,n}(d,u) = \frac{1}{n!} (d|V(S)|-|E(S)|)
	\end{align}
	we find
	\begin{align*}
	\begin{split}
	\mathbb E(|M_1|) \leq& - \FKb (\beta d, d^{-1},u) \\& +\sum_{(S,n) \in \mathcal S_d} \left( \mathrm{e}^{-\beta d} (1+ \beta)^{d-1} \right)^{|V(S)|} \beta^{n(S)+1} \frac{d|V(S)|-|E(S)|}{n!},
	\end{split}
	\\
	=& - \FKb (\beta d, d^{-1},u) +\sum_{(S,n) \in \mathcal S_d} \mathbb P(A_{S,n}) \;\;\sum_{\mathclap{y \in V(S^+)\setminus V(S)}} \; \; \mathbb P\left(N^{e_y^-}=1 \, \big| 
	\, A_{S,n}\right),
	\end{align*}
	where the last equality follows from the proof of Lemma \ref{lem:EM1-rep}.
%	where 
	Now,
	the sum on the right hand side is easily seen to be 
	the expectation of the random variable $|W_1|$ with
	\begin{align*}
	W_1:= \{x \in V(C_r^+) \setminus V(C_r) : N^{e_x^-}=1 \}.
	\end{align*}
	Fortunately, for $d(1-\mathrm{e}^{-\beta}(1+\beta))<1$, this expectation can also be calculated in a more straightforward way.
	By applying Wald's identity multiple times, we find
	\begin{align*}
		\mathbb E(|\{y \in V(C_r) : |y|=n\}|)&= \big(d(1-\mathrm{e}^{-\beta}(1+\beta))\big)^n
		\intertext{and thus}
		\mathbb E(|W_1|) &= \sum_{n=1}^\infty \mathbb E(|\{x \in W_1 : |x|=n\}|)
		\\&= \sum_{n=1}^\infty d \beta \mathrm{e}^{-\beta} \mathbb E(|\{y \in V(C_r) : |y|=n-1\}|)
		\\&= \frac{\beta d \mathrm{e}^{-\beta}}{1-d(1-\mathrm{e}^{-\beta}(1+\beta))}.
	\end{align*}
%	by applying Wald's identity:, yielding 
%	\begin{align*}\mathbb E(|W_1|) = \frac{\beta d \mathrm{e}^{-\beta}}{1-d(1-\mathrm{e}^{-\beta}(1+\beta))}.
%	\end{align*}
	%
	% ******** Drawn out calculation, do not remove ********
	%
	%More precisely, let $V_n:=\{x \in V : \dist(x,r)=n\}$ for $n \in
	%\mathbb N_0$ and set $\mu_{=1}:= d \beta \mathrm{e}^{-\beta}$ and
	%$\mu_{\geq 2}:=d(1-\mathrm{e}^{-\beta}(1+\beta))<1$ to be the expected
	%number of children $y$ of the a vertex $x$ such that the edge $\{x,y\}$
	%carries one and at least two links, respectively. Then
	%\begin{align}
	%\mathbb E[|V_n \cap V(C)|]&= \mathbb E[|V_{n-1} \cap V(C)|] \,
	%\mu_{\geq 2}\\
	%&= \ldots = ( \mu_{\geq 2} )^n \qquad \text{for }n \in \mathbb N_0,\\
	%\mathbb E[|V_n \cap W_1|] &= \mathbb E[|V_{n-1} \cap V(C)|] \,
	%\mu_{=1} \\
	%&= ( \mu_{\geq 2})^{n-1} \mu_{=1} \qquad \text{for }n \in \mathbb N.
	%\end{align}
	%Therefore,
	%\begin{align}
	%\mathbb E[|W_1|\mathbf{1}_{\Ofin}]=\sum_{n=1}^\infty \mathbb E[|V_n
	%\cap W_1|] = \mu_{=1} \sum_{n=1}^\infty (\mu_{\geq 2})^{n-1} =
	%\frac{\mu_{=1}}{1-\mu_{\geq 2}} .
	%\end{align}
	%
	% ****************
	%
	For (\textit{c}),
	let
	$0<\alpha\leq \hat \alpha$ and $d_0 \leq d \in \mathbb N$ be given. %, we 
	We
	now use that
	\begin{align*}
	\begin{split}
	\mathbb E_{\beta=\alpha/d}(|M_1|) =& F_K(\alpha,d^{-1},u) \\&+ \sum_{\substack{(S,n) \in \mathcal S_d:\\\ord(S,n)>K}} \left( \mathrm{e}^{-\alpha} \left(1+ \frac{\alpha}{d}\right)^{d-1} \right)^{|V(S)|} \left(\frac{\alpha}{d}\right)^{n(S)+1} p_{S,d}(d,u)
	\end{split}
	\end{align*}
	and estimate the sum on the right hand side. 
	By (\ref{eq:p+pbar}) and the facts that 
	$|E(S)|\leq \ord(S,n)$ and $|E(S)|\geq 1$ for $\ord(S,n) \geq 1$, we find
	\begin{align}\label{eq:asymp-estimate-1}
	\begin{split}
	&\mathbb E_{\beta=\alpha/d}(|M_1|) - F_K(\alpha,d^{-1},u) \\
	& \leq \sum_{k=K+1}^\infty \sum_{\ell=1}^k \sum_{\substack{S \subseteq T \text{ subtree}\\\text{with root }r\\\text{and }|E(S)|=\ell}} \sum_{\substack{n \in (\mathbb N_{\geq 2})^{E(S)}:\\ n(S)=k+\ell}} {\underbrace{\left( \mathrm{e}^{-\alpha} \left(1+ \frac{\alpha}{d}\right)^{d-1} \right)}_{\leq 1}}^{|V(S)|}\\
	& \hphantom{
		\leq \sum_{k=K+1}^\infty \sum_{\ell=1}^k \sum_{\substack{S \subseteq T \text{ subtree}\\\text{with root }r\\\text{and }|E(S)|=\ell}} \sum_{\substack{n \in (\mathbb N_{\geq 2})^{E(S)}:\\ n(S)=k+\ell}} {\underbrace{\left( \mathrm{e}^{-\alpha} \left(1+ \frac{\alpha}{d}\right)^{d-1} \right)}_{\leq 1}}^{|V(S)|}} 
	\mathllap{\left(\frac{\alpha}{d}\right)^{n(S)+1} \frac{d|V(S)|-|E(S)|}{n!}}.
	\end{split}
	\end{align}
	Note that, within the last expression, we may write $n(S)$, $|V(S)|$ and $|E(S)|$ in terms of $k$ and $\ell$ instead of $S$. Moreover, by \cite[Exercise 2.3.4.4-11 on p.397 and p.589]{Knuth}, the number of subtrees $S \subseteq T$ of the \dreg tree $T$ with $r \in V(S)$ and $|V(S)|=\ell+1$ is given by the $(\ell+1)^\text{th}$ $d$-Fuss-Catalan number $\frac{1}{d(\ell+1)-\ell} {{d(\ell+1)}\choose {\ell+1}}$.
	Thus, by expanding the last summation within (\ref{eq:asymp-estimate-1}) onto all $n \in (\mathbb N_0)^{E(S)}$ with $n(S)=k+\ell$, using the multinomial theorem and estimating ${{d(\ell+1)}\choose {\ell+1}}$ due to ${m \choose j} \leq m^j/j!$, we obtain
	\begin{align*}
	\begin{split}
	&\mathbb E_{\beta=\alpha/d}(|M_1|) - F_K(\alpha,d^{-1},u) \\
	&\leq  \sum_{k=K+1}^\infty \sum_{\ell=1}^k \frac{(\ell+1)^{\ell+1}}{(\ell+1)!} \frac{\alpha^{k+\ell+1}}{d^{k}} \frac{\ell^{k+\ell}}{(k+\ell)!}
	\end{split}\\
	&\leq  \, \frac{1}{d^{K+1}} \underbrace{{\hat \alpha}^{K+3}
		\sum_{k=0}^{\infty} \left(\frac{\hat \alpha}{d_0}\right)^k
		\sum_{\ell=0}^{k+K} {\hat \alpha}^{\ell}
		\frac{(\ell+2)^{\ell+2}}{(\ell+2)!} \frac{(\ell+1)^{\ell+1}}{(\ell+1)!}
		\prod_{j=1}^{k+K+1} \frac{\ell+1}{\ell+1+j} }_{=:c_K}.
	\end{align*}
	Now, by Stirling's approximation $ \frac{\ell^\ell}{\ell!}\leq
	\frac{\mathrm{e}^\ell}{\sqrt{2 \pi \ell}}\leq \mathrm{e}^\ell$ we find
	\begin{align*}
	c_K &
	\leq \hat \alpha^{K+3} \sum_{k=0}^\infty \left(\frac{\hat
		\alpha}{d_0}\right)^k \sum_{\ell=0}^{k+K} \hat \alpha^\ell
	\mathrm{e}^{2 \ell + 3} \prod_{j=1}^{k+K+1} \underbrace{\frac{\ell+1}{\ell+1+j}}_{\leq 1}\\
	&\leq \hat \alpha^{K+3} \mathrm{e}^3  \sum_{k=0}^\infty
	\left(\frac{\hat \alpha}{d_0}\right)^k \frac{(\hat \alpha
		\mathrm{e}^2)^{k+K+1}-1}{\hat \alpha \mathrm{e}^2-1}\\
	&\leq \left(\hat \alpha^2 \mathrm{e}^2\right)^{K+1} \underbrace{\frac{\hat \alpha^{2}\mathrm{e}^{3}}{\hat \alpha
		\mathrm{e}^2-1} \sum_{k=0}^\infty \left( \frac{\hat \alpha^2
		\mathrm{e}^2}{d_0}\right)^k}_{=:c<\infty}
	\end{align*}
	since we assumed that $\hat \alpha \mathrm{e}^2>1 
	$ and $d_0 > \hat \alpha^2 \mathrm{e}^2$.
\end{proof}

\begin{rem}
	The proof of Lemma \ref{lem:EM1-estimates}(\textit{a}) and (\textit{b}) shows that the given estimates correspond to estimating $\mathbb E(|M_1^-|) \leq \mathbb E(|M_1|) \leq \mathbb E(|M_1^+|)$, where
	$M_1^\pm$ are worst-case bounds on $M_1$ outside of $A_K:=\bigcup_{(S,n) \in \mathcal S_d : \ord(S,n) \leq K} A_{S,n}$. More precisely,
	we may define $M_1^\pm$ to coincide with $M_1$ on $A_K$ (i.e., on the set where we trace the propagation of $\gamma_T$ precisely%``know'' the behaviour of loops
	), while we set $M_1^-:=\emptyset$ and $M_1^+ := W_1$ otherwise. 
	This idea of tracing $\gamma_T$ whenever possible/viable and using worst-case estimates otherwise might be a practicable way to proceed in another context, too, even if there is no ``perfect'' sequence $(M_n)_n$: If one is able to construct
	worst-case bounds $(M_n^\pm)$ for the propagation of $\gamma_T$ by a construction similar to the one for $M_1$, this at least yields
	the sufficient conditions for both phases that correspond to the estimates from Lemma \ref{lem:EM1-estimates}(\textit{a}) and (\textit{b}).
\end{rem}

Apart from providing implicit but sharp phase-conditions for the parameters $(\beta,d,u)$, the estimates from Lemma \ref{lem:EM1-estimates} also allow us to find the asymptotic expansion of $\beta_c$.
 
\begin{proof}[Proof of Theorem \ref{prop:betac-asymptotics}]
 	Fix $u \in [0,1]$. 
 	Since the terms within $\FK(\alpha,h,u)$ contain the factor $h^{\ord(S,n)}$ and the only pair $(S,n)$ with $\ord(S,n)=0$ is $(S_0,n_0)$, with Example \ref{ex:order0} and $\kappa_{S_0,n_0}(d)=1$ we find
 	\begin{align*}
 	\FK(1,0,u)=1
 	\end{align*}
 	as well as 
 	\begin{align*}
 	\partial_\alpha \FK(\alpha,h,u) \big |_{\alpha=1,h=0} &= 
 	1
 	\end{align*}
 	for all $K \in \mathbb N_0$.
 	Therefore, by the implicit function theorem for analytic functions
 	(see Proposition \ref{prop:analytic-solution}) there exist analytic functions $\alpha^{(K,\pm)}$ on a common neighbourhood of $h=0$ and such that
 	\begin{align*}
 	\FK(\alpha^{(K,+)}(h),h,u)&=1
 	\intertext{and}
 	\FK(\alpha^{(K,-)}(h),h,u)+c_K h^{K+1} &= 1
 	\end{align*}
 	for sufficiently small $|h|$, where $c_K$ is chosen according to Lemma \ref{lem:EM1-estimates}(c) and $\hat \alpha := 2$. Moreover, by a corollary of the multivariate Fa\`a Di Bruno formula (see Proposition \ref{prop:analytic-solution}) the coefficients of $\alpha^{(K,\pm)}$ can be determined recursively by $\alpha_0=1$ and
 	\begin{align}
 	\begin{split}\label{eq:alphakrecursion}
 	\alpha_k=&\,\alpha_k(u)\\:=&\, - \sum_{\substack{j_0,\ldots,j_{k-1} \in \mathbb N_0:\\
 			1\leq \sum_{i=0}^{k-1} j_i \leq k,\\ j_0 + \sum_{i=1}^{k-1} i j_i = k}} \frac{\partial_h^{j_0} \partial_\alpha^{j_1+\ldots+j_{k-1}} F_k(\alpha,h,u) \big |_{\alpha=1,h=0}}{\prod_{i=0}^{k-1} j_i!} \prod_{i=1}^{k-1} \alpha_i^{j_i}, 
 	\end{split}
 	\end{align}
 	$k=1,\ldots,K$.	Here, we used that
 	\begin{align*}
 	\partial_h^j F_k (\alpha,h,u) \big |_{h=0} &= \partial_h^j \FK(\alpha,h,u) \big |_{h=0} \\&= \partial_h^j \left(\FK(\alpha,h,u) + c_K h^{K+1} \right)|_{h=0}
 	\end{align*}
 	for $j \leq k \leq K$ since these functions differ by terms containing the factor $h^{j+1}$. 
 	In particular, for $0 \leq k \leq K$, the $k^\text{th}$ coefficients of $\alpha^{(K,+)}$ and $\alpha^{(K,-)}$ coincide with $\alpha_k$ and they do not depend on the choice of $K$. This yields
 	\begin{align} \label{eq:alphakpmextension}
 	\alpha^{(K,\pm)}(h)= \sum_{k=0}^{K} \alpha_k h^k + \mathcal O(h^{K+1})
 	\end{align}
 	as $h \to 0$ with the $\mathcal O$-term of course differing for $\alpha^{(K,+)}$ and $\alpha^{(K,-)}$.
 	Furthermore, by an easy induction argument the recursion (\ref{eq:alphakrecursion}) yields that every $\alpha_k(u)$ is a polynomial in $u$ as $F_k$ is a polynomial in $u$.
 	Finally, by Lemma \ref{lem:EM1-estimates}(a), for $\beta^+=d^{-1} \, \alpha^{(K,+)}(d^{-1})$ we find that
 	\begin{align*}
 	\mathbb E_{\beta^+ 
 		}(|M_1|) \geq \FK(\alpha^{(K,+)}(d^{-1}), d^{-1}, u) =1 = \mathbb E_{\beta_c}(|M_1|)
 	\end{align*}
 	for all sufficiently large $d$. Thus, by monotonicity (see Proposition \ref{prop:betac-characterisation}) we find that 
 	\begin{align*}
 	\beta_c \leq \beta^+= d^{-1} 
 	\, \alpha^{(K,+)} 
 	(d^{-1})
 	\end{align*}
 	for those $d$. Similarly, from Lemma \ref{lem:EM1-estimates}(c), we obtain 
 	\begin{align*}
 	\beta_c \geq 
 	d^{-1} \, \alpha^{(K,-)}(d^{-1})
 	\end{align*}
 	for large $d$. Combined with (\ref{eq:alphakpmextension}), this completes the proof.	
\end{proof}

\section{Reduction to a combinatorial problem} \label{sec:combinatorics}

In this section, we are going to present a method to calculate the polynomials $p_{S,n}$ and $\bar p_{S,n}$, respectively, for every fixed $[(S,n)]\in \mathcal S$ with $E(S) \neq \emptyset$. For this purpose, 
it suffices to calculate $\mathbb E\left(\frac{\tau_S^x}{\beta} \big | A_{S,n} \right)$ for all $x \in V(S)$ (compare with the discussion in the beginning of Section \ref{sec:asymptotic-expansion}) and we will determine this quantity by 
partitioning $A_{S,n}$ into the events $A_{S,n,\nu}$ where the cluster $C_r$ is fixed to coincide with $S$ and the total number $N^e$ of links on every edge $e \in E(S)$ is given by $n(e)$, i.e.,
\begin{align*}
\begin{split}
A_{S,n,\nu} :=& A_{S,n} \cap \{\text{For all }j=1,\ldots,n(S)\text{ the }  j^\text{th}\text{ link}\text{ on }S\\& \hphantom{A_{S,n} \cap \{} \text{is of type }\star_j\text{ and occurs on the edge }e_j\} , 
\end{split}
\intertext{where}
\nu =& ((e_1,\star_1),\ldots,(e_{n(S)},\star_{n(S)})) \in \mathcal V_{S,n},\\
\mathcal V_{S,n} :=& \big \{((\epsilon_j,*_j))_{j=1}^{n(S)} : |\{j : \epsilon_j = e\}|=n(e) \text{ for all }e \in E(S)\big \}.
\end{align*}
Moreover, the time-ordering of %the links on $E(S)$ % 
the edges and types of the links
is specified by the sequence $\nu$. Here, time-ordering is understood via $\Tbeta \simeq [0,\beta)$ and in particular, the $j^\text{th}$ link is determined with respect to this order. Given $A_{S,n,\nu}$, 
determining the loop configuration is then closely related to the following task.

\medskip
\begin{cprob} \label{cprob:1}
	Fix $[(S,n)] \in \mathcal S$ with $E(S) \neq \emptyset$ and \linebreak{} $\nu=((e_j,\star_j))_{j=1}^{n(S)} \in \mathcal V_{S,n}$.
	Now, for $j=1,\ldots,n(S)$, place a link of type $\star_j$ onto the edge $e_j$ at position $\tfrac{j}{n(S)+1}\beta$, i.e. consider the deterministic link configuration $X_\nu=(X_\nu^{e,\star})_{e \in E(S),\star \in \{\cross,\dbars\}}$ with %links on $E(S) \times \Tbeta$ according to
	\begin{align}\label{eq:X-nu-def}
	X_\nu^{e,\star} = \; \; \sum_{\mathclap{\substack{j=1,\ldots,n(S):\\e_j=e \text{ and }\star_j=\star}}} \; \;
	\delta_{\frac{j}{n(S)+1}\beta}.
	\end{align}
	For this configuration, consider the loop $\gamma_S(X_\nu)$ on the tree $S$ that contains $(r,0)$ and compute the combinatorial quantities
	\begin{align}\label{eq:def_bxnu}
	b_{S,\nu}^x := \left| \left\{ j \in \{0,\ldots,n(S)\} : \{x\} \times \left( \tfrac{j}{n(S)+1}\beta,\tfrac{j+1}{n(S)+1}\beta \right) \subseteq \gamma_S(X_\nu) 
	\right\} \right|
	\end{align}
	for all $x \in V(S)$.
\end{cprob}

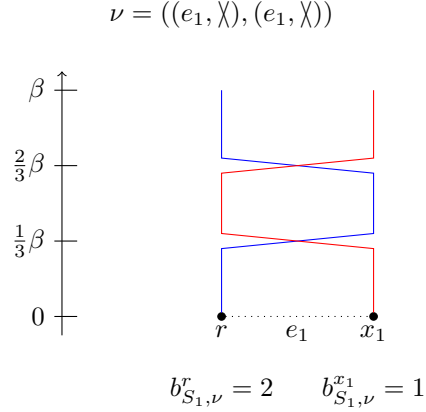
\begin{figure}
	\centering
	\begin{tikzpicture} %[scale=0.69]
	%Vertices at t=0
	\node[circle,scale=0.3,fill=black,draw] at (0,0) (x1) {};
	\node[circle,scale=0.3,fill=black,draw] at (2,0) (x2) {};
	%	\node[circle,scale=0.3,fill=black,draw] at (5,0) (x3) {};
	%	\node[circle,scale=0.3,fill=black,draw] at (7,0) (x4) {};
	%	\node[circle,scale=0.3,fill=black,draw] at (10,0) (x5) {};
	%	\node[circle,scale=0.3,fill=black,draw] at (12,0) (x6) {};				
	%Link positions
	%	\node[inner sep=0pt] at (0,1.9) (l1) {};
	%	\node[inner sep=0pt] at (2,1.2) (l2) {};
	%	\node[inner sep=0pt] at (4,3.1) (l3) {};
	%	\node[inner sep=0pt] at (4,0.6) (l4) {};
	%	\node[inner sep=0pt] at (2,2.3) (l5) {};
	%Edges connecting these vertices
	\draw[dotted] (x1) node[below]{$r$} -- node[below]{$e_1$} (x2) node[below]{$x_1$};
	%	\draw[dotted] (x3) node[below]{$r$} -- node[below]{$e_1$} (x4) node[below]{$x_1$};
	%	\draw[dotted] (x5) node[below]{$r$} -- node[below]{$e_1$} (x6) node[below]{$x_1$};				
	%		\draw[dotted] (x3) -- (x4);
	%		\draw[dotted] (x5) -- (x6);
	%Coordinate axes
	%		\draw[->]  (-.5,-.5) -- node[below] {$G=(V,E)$} +(1.5,0) ;
	%
	\node at (-2.1,4) {$\;$};
	%	\node at (-2.1,-1) {$b_\nu^x=\;$};
	\node (o) at (-2.1,0) {};
	\draw[->] (o)+(0,-.25) -- +(0,3.25);
	\draw (o)+(-.1,0) node[left] {$0$} -- +(.2,0);
	\draw (o)+(-.1,1) node[left] {$\tfrac{1}{3}\beta$} -- +(.2,1);
	\draw (o)+(-.1,2) node[left] {$\tfrac{2}{3}\beta$} -- +(.2,2);
	\draw (o)+(-.1,3) node[left] {$\beta$} -- +(.2,3);
	%	\draw (o)+(-.1,4) node[left] {$$} -- +(.2,4);
	%Draw loops
	\draw[color=myblue] (x1) -- +(0,0.9) -- +(2,1.1) -- +(2,1.9) -- +(0,2.1) -- +(0,3);
	%	\draw[color=myblue] (x1)+(0,4) -- +(0,3.1) -- +(2,3.1) -- +(2,4);
	\draw[color=myred] (x2) -- +(0,0.9) -- +(-2,1.1) -- +(-2,1.9) -- +(0,2.1) -- +(0,3);
	%
	%	\draw[color=myblue] (x3) -- +(0,0.9) -- +(2,0.9) -- +(2,0);
	%	\draw[color=myblue] (x3)+(0,4) -- +(0,3.1) -- +(2,3.1) -- +(2,4);
	%	\draw[color=myred] (x3)+(0,1.1) -- +(0,1.9) -- +(2,2.1) -- +(2,2.9) -- +(0,2.9) -- +(0,2.1) -- +(2,1.9) -- +(2,1.1) -- +(0,1.1);
	%	%
	%	\draw[color=myblue] (x5) -- +(0,0.9) -- +(2,0.9) -- +(2,0);
	%	\draw[color=myblue] (x5)+(0,4) -- +(0,3.1) -- +(2,2.9) -- +(2,2.1) -- +(0,2.1) -- +(0,2.9) -- +(2,3.1) -- +(2,4);
	%	\draw[color=myred] (x5)+(0,1.1) -- +(0,1.9) -- +(2,1.9) -- +(2,1.1) -- +(0,1.1);
	%Label vertices with number of intervals
	\node at (0,-1) {$b_{S_1,\nu}^r=2$};
	\node at (2,-1) {$b_{S_1,\nu}^{x_1}=1$};
	%	\node at (5,-1) {$2$};
	%	\node at (7,-1) {$2$};
	%	\node at (10,-1) {$3$};
	%	\node at (12,-1) {$3$};
	%Label sequences
	\node at (0,4) {$\nu=((e_1,\cross),(e_1,\cross))$};
	%	\node at (6,5) {\small $((e_1,\dbars),(e_1,\cross),(e_1,\dbars))$};
	%	\node at (11,5) {\small $((e_1,\dbars),(e_1,\dbars),(e_1,\cross))$};
	\end{tikzpicture}
	\caption{Let $S_1$ be the tree containing two vertices and one edge $e_1=\{r,x_1\}$ between them. Furthermore, let $n_1(S_1)=n_1(e_1):=2$ be the number of links on this edge and define their types by $\nu:=((e_1,\cross),(e_1,\cross))$. Then the link configuration $X_\nu$ on $E(S)$ and its corresponding loop configuration are depicted above. For every vertex $x \in V(S_1)$, one can now easily read off the number $b^x_{S_1,\nu}$ of intervals $\left(\tfrac{j}{n(S)+1}\beta,\tfrac{j+1}{n(S)+1}\beta \right)$ that the ({\color{myblue}{blue}}) loop $\gamma_{S_1}(X_\nu)$ stays at this vertex.}
	\label{fig:solve-combinatorial-problem}
\end{figure}

\medskip
\begin{rem} \label{rem:cprob} 
		One can solve the task of Combinatorial Problem \ref{cprob:1} (i.e., determine the \emph{integers} $b_{S,\nu}^x$ for all $x \in V(S)$) with the help of a computer or by drawing a sketch (see Figure \ref{fig:solve-combinatorial-problem} and the description within its caption). Unfortunately, this will take more and more computational effort as $n(S)$ and $|E(S)|$ increase.
		However, 
		note that at least
		the calculation of $b_{S,\nu}^x$ does not depend on the choice of the representative for $[(S,n)]$ if $\nu$ is adapted accordingly.
\end{rem}

The connection between Combinatorial Problem \ref{cprob:1} and the calculation of 
$p_{S,n}$ and $\bar p_{S,n}$
is established by the following lemma.

\medskip
\begin{lem}\label{lem:EofTau} For all $(S,n) \in \mathcal S_d$ with $E(S) \neq \emptyset$ and $\nu \in \mathcal V_{S,n}$ we have
	\begin{align*}
	\mathbb E\left(\frac{\tau_S^x}{\beta} \, \bigg| 
	\, A_{S,n,\nu} \right) = \frac{b_{S,\nu}^x}{n(S)+1},
	\end{align*}
	with $b_{S,\nu}^x$ given by (\ref{eq:def_bxnu}). In particular,
	\begin{align}
	p_{S,n}(d,u) =& \,\frac{1}{(n(S)+1)!} \sum_{\nu \in \mathcal V_{S,n}} u^\nu \sum_{x \in V(S)} (d-d_S^x) b_{S,\nu}^x \label{eq:pSn-2}
	\intertext{and}
	\bar p_{S,n}(d,u) =& \,\frac{1}{(n(S)+1)!} \sum_{\nu \in \mathcal V_{S,n}} u^\nu \sum_{x \in V(S)} (d-d_S^x) \bar b_{S,\nu}^x, \label{eq:pSnbar-2}
	\intertext{where for $\nu=((e_j,\star_j))_{j=1}^{n(S)}$ we set}
	u^\nu := & \,u^{|\{j : \star_j = \cross\}|} (1-u)^{|\{j : \star_j = \dbars\}|}, \notag\\
	\bar b_{S,\nu}^x :=& \,n(S)+1-b_{S,\nu}^x. \notag
	\end{align}
\begin{proof}
	To begin with, denote the positions of links on $E(S)$ by $t_1<\ldots<t_{n(S)}$ 
	and set $t_0:=0$, $t_{{n(S)}+1}:=\beta$. Moreover, fix $x \in V(S)$ and let $b^{x,j}_{S,\nu} \in \{0,1\}$, $j=0,\ldots,n(S),$ be the indicator of $\{\gamma_S \text{ contains }\{x\}\times (t_j,t_{j+1})\}$ when given $A_{S,n,\nu}$. Note that each $b^{x,j}_{S,\nu}$ is deterministic for given $S, \nu, x$ and $j$. In particular, a change of $(t_1,\ldots,t_{n(S)})$ that preserves the 
	time-ordering 
	does not change the $b^{x,j}_{S,\nu}$'s. Therefore, we find
	$b^{x,j}_{S,\nu}=b^{x,j}_{S,\nu}(X_\nu)$ with $X_\nu$ as in (\ref{eq:X-nu-def}). 
	This yields
	\begin{align*}
	\tau_S^x = \sum_{j=0}^{n(S)} b^{x,j}_{S,\nu}(X_\nu) (t_{j+1}-t_j)
	\end{align*}
	on $A_{S,n,\nu}$. Now, with respect to the conditional measure $\mathbb P(\, \cdot \, \big| \, A_{S,n,\nu})$, the vector $(t_1,\ldots,t_{n(S)})$ is
	uniformly distributed on $\{s \in \mathbb R^{n(S)} : 0<s_1<\ldots<s_{n(S)}<\beta\}$ since it is the vector of arrival times of a merged Poisson process, where the number of jumps and the assignment of these jumps to the respective subprocesses is fixed by $A_{S,n,\nu}$. Therefore, we have
	\begin{align*}
	\mathbb E\left(\frac{\tau_S^x}{\beta} \, \bigg| 
	\, A_{S,n,\nu} \right) = \sum_{j=0}^{n(S)} b_{S,\nu}^{x,j}(X_\nu) \underbrace{\mathbb E \left( \frac{t_{j+1}-t_j}{\beta} \, \bigg| \, A_{S,n,\nu} \right)}_{=\frac{1}{n(S)+1}}
	= \frac{b_{S,\nu}^x}{n(S)+1}.
	\end{align*}
	Finally, the assertions about $p_{S,n}$ and $\bar p_{S,n}$, respectively, follow if we decompose
	$ 
	A_{S,n} = \bigcup_{\nu \in \mathcal V_{S,n}} A_{S,n,\nu}
	$
	and use that $\mathbb P(A_{S,n,\nu} \, \big| \, A_{S,n}) = u^\nu \frac{n!}{n(S)!}$.
\end{proof}
\end{lem}

Note that so far we excluded the case $[(S,n)]=[(S_0,n_0)]$ within the considerations in this section since the definition of $A_{S,n,\nu}$ would need clarification to make sense for $E(S)=\emptyset$. Nevertheless, Example \ref{ex:order0} shows that (\ref{eq:pSn-2}) and (\ref{eq:pSnbar-2}) remain valid if we set $\nu=(\emptyset)$ to be the empty list and $\mathcal V_{S_0,n_0}=\{\nu\}$ as well as $b^r_{S_0,\nu}=1=u^\nu$.

Before we address the proof of Theorem \ref{thrm:sharp-phase-transition} for $d=3,4$, let us present computational results for the integers $b_{S,\nu}^x$. For the sake of a concise arrangement, we define
\begin{align*}
\mathcal V_{S,n,j} := \{((e_i,\star_i))_{i=1}^{n(S)} \in \mathcal V_{S,n} : |\{i : \star_i = \dbars\}|=j\}, \qquad j=0,\ldots,n(S)
\end{align*}
and set $D_{S,n}$ to be the $2\times (n(S)+1)$-matrix for which the $k^\text{th}$ column is given by
\begin{align*}
	(D_{S,n})_k =& \sum_{\nu \in \mathcal V_{S,n,k-1}}
	\sum_{x \in V({S})} \begin{pmatrix}
		b^x_{S,\nu} \\ d_S^x \, b^x_{S,\nu} \end{pmatrix} ,\qquad k=1,\ldots,n(S)+1. %\label{eq:Dmatrixdef}
\end{align*}

Analogously, we define $\bar D_{S,n}$ but with $b^x_{S,\nu}$ replaced by $\bar b^x_{S,\nu}$.
This yields
\begin{align*}
p_{S,n}(d,u) =&\, \frac{1}{(n(S)+1)!} \left \langle D_{S,n} \mathbf{u}^{(n(S))} ,  \begin{pmatrix}
d \\ -1 \end{pmatrix} \right \rangle
\end{align*}
%\intertext{
	and the analogous equation for $\bar p_{S,n}$, where we set
%	}
\begin{align*}
\mathbf{u}^{(n(S))} :=& \, \left( u^{n(S)} , u^{n(S)-1}(1-u) , \ldots %, u (1-u)^{n_i-1}
, (1-u)^{n(S)} \right)^\textsf{T} \in [0,1]^{n(S)+1}.
\end{align*}
Note that the entries of $D_{S,n}$ and $\bar D_{S,n}$ are integers and they do not depend on the specific choice for the representative of $[(S,n)]$, see Remark \ref{rem:cprob}.
To demonstrate how to compute their entries, let us look at an example.

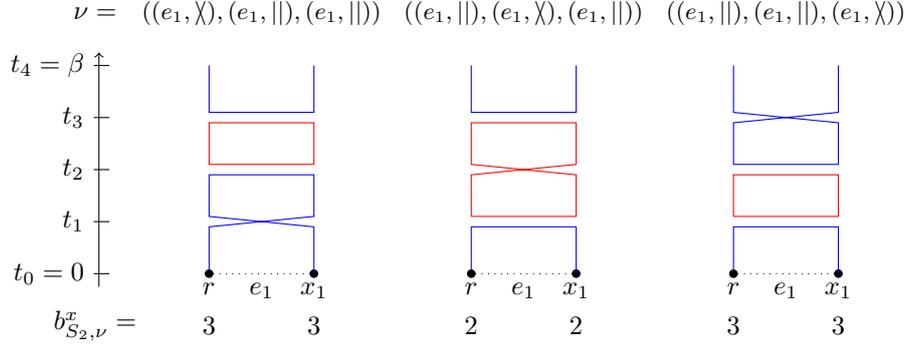
\begin{figure}
	\begin{tikzpicture}[scale=0.69]
	%Vertices at t=0
	\node[circle,scale=0.3,fill=black,draw] at (0,0) (x1) {};
	\node[circle,scale=0.3,fill=black,draw] at (2,0) (x2) {};
	\node[circle,scale=0.3,fill=black,draw] at (5,0) (x3) {};
	\node[circle,scale=0.3,fill=black,draw] at (7,0) (x4) {};
	\node[circle,scale=0.3,fill=black,draw] at (10,0) (x5) {};
	\node[circle,scale=0.3,fill=black,draw] at (12,0) (x6) {};				
	%Link positions
	%	\node[inner sep=0pt] at (0,1.9) (l1) {};
	%	\node[inner sep=0pt] at (2,1.2) (l2) {};
	%	\node[inner sep=0pt] at (4,3.1) (l3) {};
	%	\node[inner sep=0pt] at (4,0.6) (l4) {};
	%	\node[inner sep=0pt] at (2,2.3) (l5) {};
	%Edges connecting these vertices
	\draw[dotted] (x1) node[below]{$r$} -- node[below]{$e_1$} (x2) node[below]{$x_1$};
	\draw[dotted] (x3) node[below]{$r$} -- node[below]{$e_1$} (x4) node[below]{$x_1$};
	\draw[dotted] (x5) node[below]{$r$} -- node[below]{$e_1$} (x6) node[below]{$x_1$};				
	%		\draw[dotted] (x3) -- (x4);
	%		\draw[dotted] (x5) -- (x6);
	%Coordinate axes
	%		\draw[->]  (-.5,-.5) -- node[below] {$G=(V,E)$} +(1.5,0) ;
	%
	\node at (-2.1,5) {$\nu=\;$};
	\node at (-2.1,-1) {$b_{S_2,\nu}^x=\;$};
	\node (o) at (-2.1,0) {};
	\draw[->] (o)+(0,-.25) -- +(0,4.25);
	\draw (o)+(-.1,0) node[left] {$t_0=0$} -- +(.2,0);
	\draw (o)+(-.1,1) node[left] {$t_1$} -- +(.2,1);
	\draw (o)+(-.1,2) node[left] {$t_2$} -- +(.2,2);
	\draw (o)+(-.1,3) node[left] {$t_3$} -- +(.2,3);
	\draw (o)+(-.1,4) node[left] {$t_4=\beta$} -- +(.2,4);
	%Draw loops
	\draw[color=myblue] (x1) -- +(0,0.9) -- +(2,1.1) -- +(2,1.9) -- +(0,1.9) -- +(0,1.1) -- +(2,0.9) -- +(2,0);
	\draw[color=myblue] (x1)+(0,4) -- +(0,3.1) -- +(2,3.1) -- +(2,4);
	\draw[color=myred] (x1)+(0,2.1) -- +(0,2.9) -- +(2,2.9) -- +(2,2.1) -- +(0,2.1);
	\draw[color=myblue] (x3) -- +(0,0.9) -- +(2,0.9) -- +(2,0);
	\draw[color=myblue] (x3)+(0,4) -- +(0,3.1) -- +(2,3.1) -- +(2,4);
	\draw[color=myred] (x3)+(0,1.1) -- +(0,1.9) -- +(2,2.1) -- +(2,2.9) -- +(0,2.9) -- +(0,2.1) -- +(2,1.9) -- +(2,1.1) -- +(0,1.1);
	\draw[color=myblue] (x5) -- +(0,0.9) -- +(2,0.9) -- +(2,0);
	\draw[color=myblue] (x5)+(0,4) -- +(0,3.1) -- +(2,2.9) -- +(2,2.1) -- +(0,2.1) -- +(0,2.9) -- +(2,3.1) -- +(2,4);
	\draw[color=myred] (x5)+(0,1.1) -- +(0,1.9) -- +(2,1.9) -- +(2,1.1) -- +(0,1.1);
	%Label vertices with number of intervals
	\node at (0,-1) {$3$};
	\node at (2,-1) {$3$};
	\node at (5,-1) {$2$};
	\node at (7,-1) {$2$};
	\node at (10,-1) {$3$};
	\node at (12,-1) {$3$};
	%Label sequences
	\node at (1,5) {\small $((e_1,\cross),(e_1,\dbars),(e_1,\dbars))$};
	\node at (6,5) {\small $((e_1,\dbars),(e_1,\cross),(e_1,\dbars))$};
	\node at (11,5) {\small $((e_1,\dbars),(e_1,\dbars),(e_1,\cross))$};
	\end{tikzpicture}
	\caption{The three sequences $\nu \in \mathcal V_{S_3,n_3,2}$, see Example \ref{ex:matrix-entries} for a description.}
	\label{fig:entry-calc-example}
\end{figure}

\medskip
\begin{ex} \label{ex:matrix-entries}
	Consider $S_2=(\{r,x_1\},\{e_1=\{r,x_1\}\})$ and $n_2(S_2)=n_2(e_1)=3$ as well as link configurations with $j=2$ links of type $\dbars$. Then the set $\mathcal V_{S_2,n_2,2}$ consists of the three sequences $\nu$ listed on top of Figure \ref{fig:entry-calc-example}. Similar to Figure \ref{fig:solve-combinatorial-problem}, one can read off the numbers $b_{S_2,\nu}^x$ with $x \in V(S_2)$ and $\nu\in \mathcal V_{S_2,n_2,2}$ by constructing the (\textcolor{myblue}{blue}) loop $\gamma_{S_2}(X_\nu)$ 
	(see bottom line of Figure \ref{fig:entry-calc-example}). Since $d_{S_2}^r=1$ and $d_{S_2}^{x_1}=0$, the third %$3^\text{rd}$ 
	column of $D_{S_2,n_2}$ becomes 
	\begin{align*}
	(D_{S_2,n_2})_3 = \begin{pmatrix}
	3+3 & + & 2+2 & + & 3+3 \\ 3+0 & + & 2+0 & + & 3+0
	\end{pmatrix} = \begin{pmatrix}
	16 \\ 8
	\end{pmatrix}.
	\end{align*}
	All other columns of $D_{S_2,n_2}$ are determined analogously.
\end{ex} 

\begin{sidewaystable}
	\resizebox{0.9\linewidth}{!}{
		\begin{tabular}{c|c|cc|c|c|c}
			Sketch of & \multirow{2}{*}{$\ord(S,n)$}
			& \multirow{2}{*}{$n(S)$}& \multirow{2}{*}{$|V({S})|$} & \multirow{2}{*}{$\kappa_{S,n}(d)$} & \multicolumn{2}{|c}{computed results}\\
			$[(S,n)]$ & & & & & $D_{S,n}$ & $\bar D_{S,n}$ \\ \hline
			\begin{minipage}{1.5cm}\centering\begin{tikzpicture}
				\draw[opacity=0] (-.3,.2) -- (.3,.2);
				\draw[opacity=0] (-.3,-.15) -- (.3,-.15);
				\node[fill, circle, inner sep = 1pt] (r) at (0,0) {};
				\end{tikzpicture} \end{minipage} & $0$ & $0$ & $1$ & $1$ & $\displaystyle{\begin{pmatrix}
				1 \\ 0
				\end{pmatrix}}$ & $\displaystyle{\begin{pmatrix}
				0 \\ 0
				\end{pmatrix}}$ \\ \hline 
			\begin{minipage}{1.5cm}\centering\begin{tikzpicture}
				\draw[opacity=0] (-.3,.95) -- (.3,.95);
				\draw[opacity=0] (-.3,-.15) -- (.3,-.15);
				\node[fill, circle, inner sep = 1pt] (r) at (0,0) {};
				\node[fill, circle, inner sep = 1pt] (x) at (-.45,.75) {};
				\draw (r) -- node[fill=white,circle,inner sep=0pt] {$2$} (x);
				\end{tikzpicture} \end{minipage} & $1$ & $2$ & $2$ & $d$ & $\displaystyle{\begin{pmatrix}
				3 & 12 & 4 \\ 2 & 6 & 2
				\end{pmatrix}}$ & $\displaystyle{\begin{pmatrix}
				3 & 0 & 2 \\ 1 & 0 & 1
				\end{pmatrix}}$ \\ \hline
			\begin{minipage}{1.5cm}\centering\begin{tikzpicture}
				\draw[opacity=0] (-.3,.95) -- (.3,.95);
				\node[fill, circle, inner sep = 1pt] (r) at (0,0) {};
				\node[fill, circle, inner sep = 1pt] (x) at (-.45,.75) {};
				\draw (r) -- node[fill=white,circle,inner sep=0pt] {$3$} (x);
				\end{tikzpicture} \end{minipage} & $2$ & $3$ & $2$ & $d$ & $\displaystyle{\begin{pmatrix}
				8 & 24 & 16 & 4 \\ 4 & 12 & 8 & 2
				\end{pmatrix}}$ & $\displaystyle{\begin{pmatrix}
				0 & 0 & 8 & 4 \\ 0 & 0 & 4 & 2 
				\end{pmatrix}}$ \\
			\begin{minipage}{1.5cm}\centering\begin{tikzpicture}
				\draw[opacity=0] (-.3,.95) -- (.3,.95);
				\node[fill, circle, inner sep = 1pt] (r) at (0,0) {};
				\node[fill, circle, inner sep = 1pt] (x) at (-.45,.75) {};
				\node[fill, circle, inner sep = 1pt] (y) at (.45,.75) {};
				\draw (r) -- node[fill=white,circle,inner sep=0pt] {$2$} (x);
				\draw (r) -- node[fill=white,circle,inner sep=0pt] {$2$} (y);
				\end{tikzpicture} \end{minipage} & $2$ &  $4$ & $3$ & $\displaystyle{{d \choose 2}}$ & $\displaystyle{\begin{pmatrix}
				50 & 260 & 414 & 292 & 60 \\ 40 & 192 & 288 & 192 & 40 
				\end{pmatrix}}$ & $\displaystyle{\begin{pmatrix}
				40 & 100 & 126 & 68 & 30 \\ 20 & 48 & 72 & 48 & 20 
				\end{pmatrix}}$ \\ 
			\begin{minipage}{1.5cm}\centering\begin{tikzpicture}
				\draw[opacity=0] (-.3,1.7) -- (.3,1.7);
				\draw[opacity=0] (-.3,-.15) -- (.3,-.15);
				\node[fill, circle, inner sep = 1pt] (r) at (0,0) {};
				\node[fill, circle, inner sep = 1pt] (x) at (-.45,.75) {};
				\node[fill, circle, inner sep = 1pt] (y) at (-.9,1.5) {};
				\draw (r) -- node[fill=white,circle,inner sep=0pt] {$2$} (x);
				\draw (x) -- node[fill=white,circle,inner sep=0pt] {$2$} (y);
				\end{tikzpicture} \end{minipage} & $2$ & $4$ & $3$ & $d^2$ & $\displaystyle{\begin{pmatrix}
				50 & 250 & 415 & 306 & 65 \\ 39 & 190 & 287 & 206 & 44 
				\end{pmatrix}}$ & $\displaystyle{\begin{pmatrix}
				40 & 110 & 125 & 54 & 25 \\ 21 & 50 & 73 & 34 & 16
				\end{pmatrix}}$ \\ \hline
			\begin{minipage}{1.5cm}\centering\begin{tikzpicture}
				\draw[opacity=0] (-.3,.95) -- (.3,.95);
				\node[fill, circle, inner sep = 1pt] (r) at (0,0) {};
				\node[fill, circle, inner sep = 1pt] (x) at (-.45,.75) {};
				\draw (r) -- node[fill=white,circle,inner sep=0pt] {$4$} (x);
				\end{tikzpicture} \end{minipage} & $3$ & $4$ & $2$ & $d$ & $\displaystyle{\begin{pmatrix}
				5 & 40 & 40 & 20 & 4 \\ 3 & 20 & 20 & 10 & 2 
				\end{pmatrix}}$ & $\displaystyle{\begin{pmatrix}
				5 & 0 & 20 & 20 & 6 \\ 2 & 0 & 10 & 10 & 3
				\end{pmatrix}}$ \\
			\begin{minipage}{1.5cm}\centering\begin{tikzpicture}
				\draw[opacity=0] (-.3,.95) -- (.3,.95);
				\node[fill, circle, inner sep = 1pt] (r) at (0,0) {};
				\node[fill, circle, inner sep = 1pt] (x) at (-.45,.75) {};
				\node[fill, circle, inner sep = 1pt] (y) at (.45,.75) {};
				\draw (r) -- node[fill=white,circle,inner sep=0pt] {$3$} (x);
				\draw (r) -- node[fill=white,circle,inner sep=0pt] {$2$} (y);
				\end{tikzpicture} \end{minipage} & $3$ & $5$ & $3$ & $d(d-1)$ & $\displaystyle{\begin{pmatrix}
				102 & 750 & 1396 & 1348 & 624 & 104 \\ 80 & 528 & 960 & 900 & 408 & 68
				\end{pmatrix}}$ & $\displaystyle{\begin{pmatrix}
				78 & 150 & 404 & 452 & 276 & 76 \\ 40 & 72 & 240 & 300 & 192 & 52
				\end{pmatrix}}$ \\ 
			\begin{minipage}{1.5cm}\centering\begin{tikzpicture}
				\draw[opacity=0] (-.3,1.7) -- (.3,1.7);
				\node[fill, circle, inner sep = 1pt] (r) at (0,0) {};
				\node[fill, circle, inner sep = 1pt] (x) at (-.45,.75) {};
				\node[fill, circle, inner sep = 1pt] (y) at (-.9,1.5) {};
				\draw (r) -- node[fill=white,circle,inner sep=0pt] {$3$} (x);
				\draw (x) -- node[fill=white,circle,inner sep=0pt] {$2$} (y);
				\end{tikzpicture} \end{minipage} & $3$ & $5$ & $3$ & $d^2$ & $\displaystyle{\begin{pmatrix}
				114 & 774 & 1418 & 1382 & 646 & 110 \\ 92 & 560 & 982 & 932 & 426 & 72
				\end{pmatrix}}$ & $\displaystyle{\begin{pmatrix}
				66 & 126 & 382 & 418 & 254 & 70 \\ 28 & 40 & 218 & 268 & 174 & 48
				\end{pmatrix}}$ \\
			\begin{minipage}{1.5cm}\centering\begin{tikzpicture}
				\draw[opacity=0] (-.3,1.7) -- (.3,1.7);
				\node[fill, circle, inner sep = 1pt] (r) at (0,0) {};
				\node[fill, circle, inner sep = 1pt] (x) at (-.45,.75) {};
				\node[fill, circle, inner sep = 1pt] (y) at (-.9,1.5) {};
				\draw (r) -- node[fill=white,circle,inner sep=0pt] {$2$} (x);
				\draw (x) -- node[fill=white,circle,inner sep=0pt] {$3$} (y);
				\end{tikzpicture} \end{minipage} & $3$ & $5$ & $3$ & $d^2$ & $\displaystyle{\begin{pmatrix}
				84 & 696 & 1386 & 1430 & 710 & 126 \\ 62 & 482 & 950 & 980 & 490 & 88
				\end{pmatrix}}$ & $\displaystyle{\begin{pmatrix}
				96 & 204 & 414 & 370 & 190 & 54 \\ 58 & 118 & 250 & 220 & 110 & 32
				\end{pmatrix}}$ \\
			\begin{minipage}{1.5cm}\centering\begin{tikzpicture}
				\draw[opacity=0] (-.3,1.7) -- (.3,1.7);
				\node[fill, circle, inner sep = 1pt] (r) at (0,0) {};
				\node[fill, circle, inner sep = 1pt] (x) at (-.45,.75) {};
				\node[fill, circle, inner sep = 1pt] (y) at (.45,.75) {};
				\node[fill, circle, inner sep = 1pt] (z) at (-.9,1.5) {};
				\draw (r) -- node[fill=white,circle,inner sep=0pt] {$2$} (x);
				\draw (r) -- node[fill=white,circle,inner sep=0pt] {$2$} (y);
				\draw (x) -- node[fill=white,circle,inner sep=0pt] {$2$} (z);
				\end{tikzpicture} \end{minipage} & $3$ & $6$ & $4$ & $d^2(d-1)$ & $\displaystyle{\begin{pmatrix}
				1162 & 9268 & 25614 & 36910 & 28483 & 11254 & 1539 \\ 1127 & 8186 & 21292 & 29264 & 21899 & 8534 & 1182
				\end{pmatrix}}$ & $\displaystyle{\begin{pmatrix}
				1358 & 5852 & 12186 & 13490 & 9317 & 3866 & 981 \\ 763 & 3154 & 7058 & 8536 & 6451 & 2806 & 708
				\end{pmatrix}}$ \\
			\begin{minipage}{1.5cm}\centering\begin{tikzpicture}
				\draw[opacity=0] (-.3,1.7) -- (.3,1.7);
				\node[fill, circle, inner sep = 1pt] (r) at (0,0) {};
				\node[fill, circle, inner sep = 1pt] (x) at (-.45,.75) {};
				\node[fill, circle, inner sep = 1pt] (y) at (-.9,1.5) {};
				\node[fill, circle, inner sep = 1pt] (z) at (0,1.5) {};
				\draw (r) -- node[fill=white,circle,inner sep=0pt] {$2$} (x);
				\draw (x) -- node[fill=white,circle,inner sep=0pt] {$2$} (y);
				\draw (x) -- node[fill=white,circle,inner sep=0pt] {$2$} (z);
				\end{tikzpicture} \end{minipage} & $3$ & $6$ & $4$ & $\displaystyle{d\cdot {d\choose 2}}$ & $\displaystyle{\begin{pmatrix}
				1232 & 9996 & 25748 & 37964 & 28888 & 12320 & 1652 \\ 1090 & 8388 & 20918 & 29840 & 22046 & 9228 & 1258
				\end{pmatrix}}$ & $\displaystyle{\begin{pmatrix}
				1288 & 5124 & 12052 & 12436 & 8912 & 2800 & 868 \\ 800 & 2952 & 7432 & 7960 & 6304 & 2112 & 632
				\end{pmatrix}}$ \\
			\begin{minipage}{1.5cm}\centering\begin{tikzpicture}
				\draw[opacity=0] (-.3,1.7) -- (.3,1.7);
				\node[fill, circle, inner sep = 1pt] (r) at (0,0) {};
				\node[fill, circle, inner sep = 1pt] (x) at (-.45,.75) {};
				\node[fill, circle, inner sep = 1pt] (y) at (-.9,1.5) {};
				\node[fill, circle, inner sep = 1pt] (z) at (-1.35,2.25) {};
				\draw (r) -- node[fill=white,circle,inner sep=0pt] {$2$} (x);
				\draw (x) -- node[fill=white,circle,inner sep=0pt] {$2$} (y);
				\draw (y) -- node[fill=white,circle,inner sep=0pt] {$2$} (z);
				\end{tikzpicture} \end{minipage} & $3$ & $6$ & $4$ & $d^3$ & $\displaystyle{\begin{pmatrix}
				1022 & 8456 & 24204 & 36386 & 28943 & 11634 & 1615 \\ 923 & 7238 & 19986 & 28854 & 22248 & 8876 & 1245
				\end{pmatrix}}$ & $\displaystyle{\begin{pmatrix}
				1498 & 6664 & 13596 & 14014 & 8857 & 3486 & 905 \\ 967 & 4102 & 8364 & 8946 & 6102 & 2464 & 645
				\end{pmatrix}}$ \\
			\begin{minipage}{1.5cm}\centering\begin{tikzpicture}
				\draw[opacity=0] (-.3,1.2) -- (.3,1.2);
				\draw[opacity=0] (-.3,-.15) -- (.3,-.15);
				\node[fill, circle, inner sep = 1pt] (r) at (0,0) {};
				\node[fill, circle, inner sep = 1pt] (x) at (-.6,1) {};
				\node[fill, circle, inner sep = 1pt] (y) at (0,1) {};
				\node[fill, circle, inner sep = 1pt] (z) at (.6,1) {};
				\draw (r) -- node[fill=white,circle,inner sep=0pt] {$2$} (x);
				\draw (r) -- node[fill=white,circle,inner sep=0pt] {$2$} (y);
				\draw (r) -- node[fill=white,circle,inner sep=0pt] {$2$} (z);
				\end{tikzpicture} \end{minipage} & $3$ & $6$ & $4$ & $\displaystyle{{d \choose 3}}$ & $\displaystyle{\begin{pmatrix}
				1260 & 10332 & 26208 & 37764 & 27900 & 11592 & 1512 \\ 1134 & 8568 & 21186 & 29376 & 21186 & 8568 & 1134
				\end{pmatrix}}$ & $\displaystyle{\begin{pmatrix}
				1260 & 4788 & 11592 & 12636 & 9900 & 3528 & 1008 \\ 756 & 2772 & 7164 & 8424 & 7164 & 2772 & 756
				\end{pmatrix}}$ \\ 
		\end{tabular}}
		\caption{The considered prototypes of edge-weighted rooted trees $[(S,n)] \in \mathcal S$ are sorted by their order $\ord(S,n)$ and displayed together with algorithmically computed results $D_{S,n}$, $\bar D_{S,n}$. For each pair $[(S,n)]$, the values of $n(e)$ are attached to their corresponding edge $e$ and the root is always depicted as the bottom vertex.}
		\label{tab:exploration-scheme-results}
	\end{sidewaystable}

Similar to Example \ref{ex:matrix-entries}, we have determined the matrices $D_{S,n}$ for all $[(S,n)] \in \mathcal S$ with $\ord(S,n) \leq 5$ and (for $\ord(S,n) \leq 3$) they are listed within Table \ref{tab:exploration-scheme-results}.	
Together with the corresponding multiplicities $\kappa_{S,n}(d)$ that are also listed in this table, this allows us to calculate $\FK(\alpha,h,u)$ and $\FKb(\alpha,h,u)$ for $0 \leq K \leq 5$ and all $(\alpha,h,u)$. In particular, we may now compute the coefficients $\alpha_k(u)$ using (\ref{eq:alphakrecursion}) and the results are given in Table \ref{tab:alphak}. Furthermore, we may now complete the proof of Theorem \ref{thrm:sharp-phase-transition}.

\begin{proof}[Proof of Theorem \ref{thrm:sharp-phase-transition} for $d=3,4$]
	By Proposition \ref{prop:betac-characterisation}, it is sufficient to show that $\mathbb E_{\beta=d^{-1/2}}(|M_1|)>1$. Moreover, by Lemma \ref{lem:EM1-estimates}(\textit{a}), a sufficient condition for the latter statement is $F_5(d^{-1/2} d, d^{-1},u)>1$ and one sees that this holds for $d=3,4$ (compare Figure \ref{fig:d34plots}).
\end{proof}

\begin{figure}
	\centering
	\includegraphics[width=0.8\textwidth]{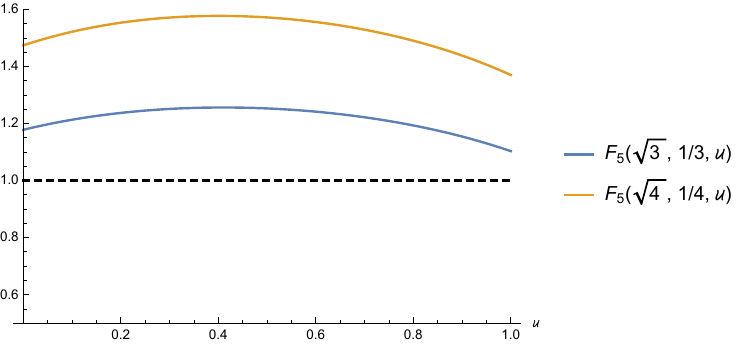}
	\caption{Plot of $F_5(d^{1/2},d^{-1},u)$ as a function of $u$ for $d=3,4$. In particular, both graphs are strictly above $1$ uniformly in $u$.}
	\label{fig:d34plots}
\end{figure}

%%
%% ************* Remark, do not remove! *************
%%
%%More precisely, we calculate
%%\begin{align}
%%& \, - 1 + 2 D \underbrace{\ln\left(1+\tfrac{1}{D}\right)}_{\geq \frac{1/D}{1/D+1}} + (D^2-1)\frac{1}{1+\tfrac{1}{D}} \frac{-1}{D^2} + \frac{1}{D} \\
%%\geq & \, \frac{1}{1/D+1} \left( -\frac{1}{D}-1 + 2 - 1 + \frac{1}{D^2} + \frac{1}{D^2} + \frac{1}{D} \right) >0.
%%\end{align}
%%Therefore, (\ref{eq:EM1greater1-1}) is monotone increasing in $d$.
%%
%% ****************************************************
%%

\medskip
\begin{rem}[Concerning a sharp phase transition for $d=2$] \label{rem:d2transition}~\\
	In Theorem \ref{thrm:sharp-phase-transition}, the case $d=2$ of the binary tree is excluded. In this boundary case we are missing two crucial properties: On the one hand, we need to find a sufficiently large $\beta^*>0$ (possibly depending on $u$) such that we can show $\mathbb E_{\beta^*}(|M_1|)>1$ for all $u$ by an appropriate estimate. On the other hand, $\beta^*$ needs to be small enough that $(0,\beta^*] \ni \beta \mapsto \mathbb E_{\beta}(|M_1|)$ is strictly increasing.\\
	Note that, for $d=2$, we would need to choose $\beta^*>d^{-1/2}$ since a numerical evaluation of $\bar F_5$ yields $\mathbb E_{\beta=d^{-1/2}}(|M_1|)\leq1$ for $d=2$ and all $u$. Unfortunately, this means that our proof of monotonicity (see Proposition \ref{prop:betac-characterisation}) fails as $f_{S_0,n_0}(\beta)$ is decreasing for $\beta>d^{-1/2}$ and thus, the representation of $\mathbb E(|M_1|)$ given in Lemma \ref{lem:EM1-rep} becomes a sum where some terms are increasing and some are decreasing. \\
	Nevertheless, up to $\beta^*=1$ and for all $[(S,n)]\in \mathcal S$ excluding $[(S_0,n_0)]$, the map $\beta \mapsto f_{S,n}(\beta)$ remains strictly increasing and numerical results suggest that $\mathbb E_\beta(|M_1|)$ remains increasing up to this value, too. Moreover, for $d=2$ and $\betaplus=1$, we find that $F_5(\betaplus d,d^{-1}, u)>1$ holds for a large range of $u$ including $u=\tfrac 12$. For the missing values of $u$ (in particular for $u=0,1$) an approximation by $\FK$ with $K=9$ should suffice to show that there also is a phase of infinite loops. 
\end{rem}

\section*{Acknowledgements}

The research of BL was supported by the Alexander von Humboldt Foundation.

%\nocite{*}
\bibliography{literatur}
\bibliographystyle{plain}
\addcontentsline{toc}{section}{References}

\appendix
\section{Analytic equations and their solutions}
Suppose that we are given an equation $f(x,y)=0$ and some $x_0,y_0 \in \mathbb R$ with $f(x_0,y_0)=0$. Then the classical implicit function theorem gives a sufficient condition such that one may find a unique solution $y=g(x)$ to this equation in a neighbourhood of $x_0$. If the function $f$ is in fact analytic, then $g$ can be shown to be analytic, too. Moreover, there exists an explicit recursion (involving the derivatives of $f$) to determine the coefficients of the series expansion of $g$ around $x_0$. 

\medskip
\begin{prop} \label{prop:analytic-solution}~\\
	Let $f \colon U \to \mathbb R$ be an analytic function in a neighbourhood $U \subseteq \mathbb R^2$ of $(x_0,y_0)\in U$. If $f(x_0,y_0)=0$ and $D_2 f(x_0,y_0) \neq 0$, then there exists a neighbourhood $V$ of $x_0$ and an analytic function $g\colon V \to \mathbb R, \, g(x)=\sum_{i=0}^\infty a_i (x-x_0)^i$ with %$g(x_0)=y_0$ and 
	$f(x,g(x))=0$ for all $x \in V$. Moreover, $a_0=y_0$ and for $k=1,2,\ldots$ we have
	\begin{align*}
	a_k = - \sum 
	\frac{(D_1^{j_0} D_2^{j_1+\ldots+j_{k-1}}% D_1^{\sum_{i=1}^{k-1}j_i} 
	f )(x_0,a_0)}{(D_2 f)(x_0,a_0)\, \prod_{i=0}^{k-1} j_i!} \prod_{i=1}^{k-1} a_i^{j_i},
	\end{align*}
	where the sum runs over all $j_0,\ldots,j_{k-1}\in \mathbb N_0$ such that
	\begin{align*}
	1 \leq \sum_{i=0}^{k-1} j_i \leq k \quad \text{and} \quad j_0 + \sum_{i=1}^{k-1} i j_i = k.
	\end{align*}
	
	\begin{proof}
	By the implicit function theorem for analytic functions (see e.g. \cite[Theorem 2.3.1]{Krantz}), there exists an analytic function $g$ in some neighbourhood $V$ of $x_0$ with $a_0=g(x_0)=y_0$ and $f(x,g(x))=0$ for all $x \in V$. Thus, on the one hand, we have
	\begin{align} \label{eq:faadibruno0}
	\frac{1}{k!} \frac{\mathrm{d}^k}{\mathrm{d}x^k} f(x,g(x)) \big |_{x=x_0} = \frac{1}{k!} \frac{\mathrm{d}^k}{\mathrm{d}x^k} 0 \big |_{x=x_0} =0
	\end{align}
	for all $k \in \mathbb N$. On the other hand, the multivariate version of Fa\`{a} di Bruno's formula (see e.g. \cite[Cor 2.11]{FaaDiBruno}) yields
	\begin{align}
	\begin{split}
	\frac{1}{k!}\frac{\mathrm{d}^k}{\mathrm{d}x^k} f(x,g(x)) \big |_{x=x_0} &= \sum_{\substack{\lambda, \mu \in \mathbb N_0 : \\ 1 \leq \lambda + \mu \leq k}} \sum_{p(k,\lambda,\mu)} D_1^\lambda D_2^\mu f(x_0,g(x_0)) \\
	& \qquad \qquad \prod_{i=1}^k \frac{\left( \operatorname{id}^{(i)}(x_0)\right)^{\ell_i}\left(g^{(i)}(x_0)\right)^{j_i} }{ \ell_i! j_i! (i!)^{\ell_i+j_i}},
	\end{split} \label{eq:faadibruno1}
	\intertext{where}
	\begin{split}
	p(k,\lambda,\mu) &= \{\ell_1,\ldots,\ell_k,j_1,\ldots,j_k \geq 0 : \\
	& \qquad  \sum_{i=1}^k \ell_i= \lambda, \sum_{i=1}^k j_i  = \mu, \sum_{i=1}^k i (\ell_i+j_i)=k\}. \notag
	\end{split}
	\end{align}
	Since $\operatorname{id}^{(i)}(0)=0$ for all $i \geq 2$, the summands in (\ref{eq:faadibruno1}) with $\ell_i>0$ for some $i \geq 2$ vanish. For all other summands we have $\ell_2=\ldots=\ell_k=0, \ell_1=\lambda-\sum_{i=2}^k \ell_i=\lambda$ and $k=\sum_{i=1}^k i(\ell_i+j_i) = \lambda + \sum_{i=1}^k i j_i$. We now use that $g(x_0)=a_0$ and $g^{(i)}(x_0) =i! \, a_i$ to obtain
	\begin{align*}
	\begin{split}
	\frac{1}{k!} \frac{\mathrm{d}^k}{\mathrm{d}x^k} f(x,g(x)) \big |_{x=x_0} 
	=&\sum_{\substack{\lambda, \mu \in \mathbb N_0 : \\ 1 \leq \lambda + \mu \leq k}} \sum_{\substack{j_1,\ldots,j_k \geq 0 : \\\sum_{i=1}^k j_i = \mu,\\\lambda + \sum_{i=1}^k i j_i = k}}  D_1^\lambda D_2^\mu f(x_0,a_0) \; \frac{1}{\lambda!}  \prod_{i=1}^k \frac{a_i^{j_i}}{j_i!} .
	\end{split}
	\label{eq:faadibruno2}
	\end{align*}
	Let us investigate those summands within the right hand side of this equation with $j_k \geq 1$. Then $k \geq k-\lambda =\sum_{i=1}^k i j_i \geq k j_k \geq k$. In particular, all these inequalities are equalities, actually. Therefore, $j_k \geq 1$ implies 
	\begin{align*}
	\lambda=0=j_1,\ldots,j_{k-1} \quad \text{and} \quad \mu=j_k=1.
	\end{align*}
	Thus, there is only one summand with $j_k \neq 0$, namely the one with these parameters and it is given by
	$
	D_2 f(x_0,a_0) \, a_k
	$.
	For all other summands we have $j_k = 0$ and, in particular, $\frac{a_k^{j_k}}{j_k!}=1$. Moreover, these other summands fulfill $\mu = \sum_{i=1}^{k-1} j_i$. Thus, by writing $j_0:=\lambda$ we find
	\begin{align*}
	\begin{split}
	\frac{1}{k!} \frac{\mathrm{d}^k}{\mathrm{d}x^k} f(x,g(x)) \big |_{x=x_0} =& D_2 f(x_0,a_0) \,a_k \\
	+& \sum_{\substack{j_0,\ldots,j_{k-1} \geq 0 : \\1 \leq j_0 + \sum_{i=1}^{k-1} j_i \leq k,\\ j_0 + \sum_{i=1}^{k-1} i j_i = k}} \frac{(D_1^{j_0} D_2^{j_1+\ldots+j_{k-1}}
		f )(x_0,a_0)}{\prod_{i=0}^{k-1} j_i!}  \prod_{i=1}^{k-1} a_i^{j_i}.
	\end{split}
	\end{align*}
	Together with (\ref{eq:faadibruno0}), this yields the assertion.
	\end{proof}
\end{prop}

\end{document}